\documentclass[10pt,letterpaper,reqno]{amsart}
\usepackage[polutonikogreek, english]{babel}
\usepackage{enumerate}

\usepackage[all]{xy}
\newdir{|>}{!/4.5pt/\dir{|}
*:(1,-.2)\dir^{>}
*:(1,+.2)\dir_{>}}
\newdir{(>}{{}*!/-8pt/\dir{>}}
\newdir^{((}{{}*!/-5pt/\dir^{(}}
\xyoption{2cell}
\xyoption{curve}
\xyoption{rotate}
\UseAllTwocells

\usepackage{bbm}
\usepackage{mathrsfs}
\usepackage{amsfonts}
\usepackage{bm}
\usepackage{amsmath,amssymb,amsthm,stmaryrd}

\usepackage{hyperref}

\usepackage{bussproofs}

\usepackage[numbers, sort&compress]{natbib}
\usepackage{mathbbol}

\swapnumbers

\newcommand{\nocontentsline}[3]{}
\newcommand{\tocless}[2]{\bgroup\let\addcontentsline=\nocontentsline#1*{#2}\egroup}

\newtheorem{teorema}{Theorem}[section]

\newtheorem*{teoA}{Theorem A}
\newtheorem*{teoB}{Theorem B}
\newtheorem*{teoC}{Theorem C}
\newtheorem*{teoD}{Theorem D}
\newtheorem*{teoE}{Theorem E}
\newtheorem{cor}[teorema]{Corollary}

\newtheorem*{que}{Question}

\newtheorem{lema}[teorema]{Lemma}

\newtheorem{prop}[teorema]{Proposition}

\theoremstyle{definition}

\newtheorem{obse}[teorema]{Remark}

\newcommand{\overbar}[1]{\mkern 1.5mu\overline{\mkern-1.5mu#1\mkern-1.5mu}\mkern 1.5mu}

\newcommand{\E}{\mathcal{E}}

\newcommand{\finord}{\mathbf{FinOrd}}
\newcommand{\N}{\mathbb{N}}

\newcommand{\teo}{\mathrm{Th}}

\def\C{\mathscr C}

\newcommand{\subob}{\mathit{\Omega}}

\DeclareMathOperator{\dso}{\Gamma}
\DeclareMathOperator{\dec}{Dec}

\DeclareMathOperator{\ev}{ev}
\DeclareMathOperator{\sh}{Sh}
\newcommand{\set}{\mathbf{Set}}
\DeclareMathOperator{\Fix}{Fix}

\DeclareMathOperator{\Clim}{Colim}

\DeclareMathOperator{\Nat}{Nat}

\newcommand{\Z}{\mathbb{Z}}

\begin{document}
\title{Tininess and right adjoints to exponentials}
\author[Ruiz]{Enrique Ruiz Hern\'andez${}^\dagger$}
\address[$\dagger$]{Centro de Investigaci\'on en Teor\'ia de Categor\'ias y sus Aplicaciones, A.C. M\'e\-xi\-co.}
\curraddr{}
\email{e.ruiz-hernandez@cinvcat.org.mx}

\author[Sol\'orzano]{Pedro Sol\'orzano${}^\star$}
\address[$\star$]{Instituto de Matem\'aticas, Universidad Nacional Aut\'onoma de M\'exico, Oaxaca de Ju\'a\-rez, M\'e\-xi\-co.}
\curraddr{}
\email{pedro.solorzano@matem.unam.mx}
\thanks{($\star$) {SECIHTI-UNAM} Research fellow.}
\date{\today}
\begin{abstract}
Objects $T$ whose exponential functor $(-)^T$ admits a right adjoint $(-)_T$ are known under different names. The fact that they exist, yet that the only set that satisfies this in the category of sets is the singleton made Lawvere suggest they ought to be ``amazingly tiny''---hence Lawvere's acronym ``A.T.O.M.''  

This report explores how intuitively tiny any such object is. Evidences both in favor and to the contrary are produced by looking at their categorical behavior (subobjects, quotients, retracts, etc) when the ambient category is a topos.

The topological behavior (connectedness, contractibility, connected components, etc) of both $T$ and $(-)_T$ is further analyzed in toposes that satisfy certain precohesive conditions over their decidable objects, where this tininess is tested against parts of Lawvere's foundational proposal for Synthetic Differential Geometry. 

\end{abstract}

\maketitle

\tocless\section{Motivation}
In the traditional models for Synthetic Differential Geometry, the tangent structure is not only represented $(-)^T$ but it is further observed to have an amazing right adjoint $(-)_T$.   Objects $T$ that satisfy this are sometimes called infinitesimal, tiny, or {\em atomic} (e.g. \citet{LawToposLaw}).  The latter name will be used throughout this report. Having certain atomic objects allows for several convenient descriptions of classical objects (see the works by \citet{MR2244115}, \citet{MR1711569}, \citet{MR1083355}, \citet{aM1999}, \citet{MR1930322}, etc.).  

In Lawvere (2011) and earlier, it is suggested that $A^T$ should be isomorphic to $A$ for an object $A$ in an appropriate subcategory of ``discrete spaces''; i.e. spaces of Non becoming.  By reverse engineering, \citet{MR890028} beautifully formalizes this. 

As a natural continuation from the authors' previous work \cite{MR4922622}, this project aims to understand the cohesive properties of an atomic $T$ and of $(-)_T$.  Already in \cite{MR1711569} it is observed that the points of $T$ determine certain enrichments. The atomic objects studied in Synthetic Differential Geometry all have the additional property of being single-pointed, yet there are examples of atomic objects with multiple points.  

A topos $\E$ is precohesive over a topos $\mathcal S$ if there is a string of adjoints 
\[\xymatrix{
\ar@<-3.5ex>@{}[rr]|{\dashv}\ar@<-3.5ex>@{}[rr]|(.15){\dashv}\ar@<-3.5ex>@{}[rr]|(.83){\dashv} & \E\ar@/^.7pc/[d]^{F_\ast}\ar@/_2.5pc/[d]_{F_!} & \\
& \mathcal S\ar@/^.5pc/[u]^{F^\ast}\ar@/_2.5pc/[u]_{F^!} &
}\]
with $F^*$ fully faithful, the counit of $F^*\dashv F_*$ monic, and such that $F_!$ preserves finite products.   

Throughout this report, a {\em McLarty topos} will be a 2-valued topos, where supports split and which is precohesive over a boolean base (see \citet{MR0877866, MR0925615}).  Examples include any model for ETCS (with or without Axiom of Choice) as well as any topos precohesive over the category of sets (See also \citet{MR4928709}). They are naturally the ambient category for models of Synthetic Differential Geometry (SDG, see \cite{MR0925615}) In such case, the precohesive context will be denoted by 
\begin{equation}\label{E:McLarty}
\vcenter{\xymatrix{
\ar@<-3.5ex>@{}[rr]|{\dashv}\ar@<-3.5ex>@{}[rr]|(.26){\dashv}\ar@<-3.5ex>@{}[rr]|(.73){\dashv} & \E\ar@/^.7pc/[d]^{\Gamma}\ar@/_2.5pc/[d]_{\Pi} & \\
& \dec(\E)\ar@/^.7pc/[u]^{\mathcal{I}}\ar@/_2.5pc/[u]_{\Lambda} &
}}
\end{equation}
where $\Pi$ represents connected components (and, for example, an object is connected if its image is terminal), $\mathcal I$ is the inclusion of decidable objects of $\E$ into $\E$, $\Gamma$ is the {\em set} of points (i.e. every global element of $X$ factors through $\mathcal I\Gamma(X)$) and $\Lambda$ is equivalent to the inclusion of $\neg\neg$-sheaves (See \cite{MR0925615} and \cite{MR4928709, MR4922622} for more details).

\

\tocless
\section{Main results}
In a cartesian closed category finite products of atomic objects are atomic and that initial objects are not atomic unless the category is trivial. For any arrow $f:S\to S'$ between atomic objects in a cartesian closed category, one can produce, {\em mutatis mutandis}, a natural transformation 
\begin{equation}\label{E:Subf}
(-)_f:(-)_S\to(-)_{S'},
\end{equation}
in the same way $(-)^f$ is constructed. Similarly, for a given object $Y$ one obtains a functor $Y_{(-)}$ from the full subcategory of atomic objects back into the ambient category.  Any object $Y$ is thus a retract of $Y_T$ for any atomic $T$ with a point---by setting alternatively $S$ and $S'$ equal to $T$ and $1$, from which it follows that $(-)_T$ is faithful.

\begin{teoA}\hypertarget{T:TeorA} Let $\E$ be a topos with an atomic object $T$ with at least one point $p:1\to T$ and let $\varepsilon$ be the counit of $(-)^T\dashv (-)_T$. For any object $X$, the arrow $j_{p, X}:X_T\rightarrow (\subob_T)^X$, given in terms of the internal language as follows

\begin{equation}\label{E:GenSingleton}
(\xi\mapsto (x\mapsto X_p(x)=_T\xi)),
\end{equation}
is a monomorphism making the following diagrams commute:

\begin{equation}
\vcenter{
\xymatrix{X\ar[d]_{X_p}\ar[rr]^{\{-\}_X} && \subob^X\ar[d]^{(\subob_p)^X} \\
X_T\ar[rr]_{j_{p,X}}&& (\subob_T)^X}}
\end{equation}

\begin{equation}
\vcenter{
\xymatrix{(X_T)^T\ar[d]_{\varepsilon_{X}}\ar[rr]^(.4){(j_{p,X})^T} && ((\subob_T)^X)^T\cong ((\subob_T)^T)^X\ar[d]^{(\varepsilon_{\subob})^X} \\
X\ar[rr]_{\{-\}_X}&& \subob^X}}
\end{equation}
wherefrom $\varepsilon_{X}$ factors through $j_{p,X}$ for any point $p:1\to T$.  Moreover, if $S$ is another atomic object and $f:T\to S$, the following diagram commutes. 
\begin{equation}
\vcenter{\xymatrix{X_T\ar[d]_{j_{p,X}}\ar[r]^{X_f} &X_S\ar[d]^{j_{f\circ p,X}} \\
(\subob_T)^X\ar[r]_{(\subob_f)^X}& (\subob_S)^X}}.
\end{equation}
In particular, for $T=1$ and $f=1$, it follows that $j_{1,X}=\{-\}_X$. Because of this, the $j_{p,X}$ can be regarded as generalized singletons. 
\end{teoA}
Another consequence of the existence of the natural transformations of \eqref{E:Subf} is the following result. Yetter observed that in presheaf categories every representable is atomic, provided that the underlying category has binary products \cite{MR890028}. Madanshekaf further shows that if the underlying category has finite products, every atomic object is a retract of a representable functor \cite{aM1999}. 

\begin{teoB}\hypertarget{T:TeorB} In any topos, retractions of atomic objects are atomic. Furthermore, examples are produced of nonatomic subobjects and quotients of atomic objects. 
\end{teoB}
The proof presented here of \hyperlink{T:TeorB}{Theorem B} explicitly constructs the amazing right adjoint. This result could be compared with the discussion of \citet{MR2177301} (see Proposition 5.25 therein) on small-projective objects, where the preservation of colimits by the exponential functor is seen to be inherited by retracts.

In the context of precohesion, the relation between atomic objects in both toposes is analyzed. In the next result, its first part could be stated more generally for exponential ideals, yet the proof provided for the second part does require locality. 
\begin{teoC}\hypertarget{T:TeorC} Let $\E$ and $\mathcal S$ be toposes.  Let $F:\E\rightarrow\mathcal S$ be precohesive, then
\begin{enumerate}
\item the functor $F^\ast$ reflects atomic objects; and
\item the functor $F_!$ preserves atomic objects.
\end{enumerate}
Furthermore, examples are provided that show that $F_!$ need not reflect atomic objects and that $F_\ast$ may not preserve nor reflect atomicity. 
\end{teoC}
In McLarty toposes, both $F^\ast$ and $F^!$ do preserve atomic objects by \hyperlink{T:TeorD}{Theorem D} below for rather trivial reasons. In general, as suggested by \citet[Proposition 3]{MR3288694}, one should not expect $F^!$ to reflect atomic objects, yet an explicit construction within the context of precohesion is not at present available.  

Even in McLarty toposes, both $\Gamma (X^T)$ and $\Pi(X^T)$ can be very different from $\Gamma(X)$ or $\Pi(X)$.  The following result establishes a more rigid behavior for the amazing right adjoint, making it a tad more {\em amazing}. 

\begin{teoD} \hypertarget{T:TeorD}The atomic objects in boolean McLarty toposes are terminal and, therefore, in any McLarty topos atomic objects are connected.  

In any McLarty topos, for any atomic $T$ and for an arbitrary object $Y,$ 
\begin{equation}\label{E:GammasubTandPisubT}
\Gamma(Y_T)\cong\Gamma(Y)\qquad\text{and}\qquad\Pi(Y_T)\cong\Pi(Y).
\end{equation}
\end{teoD}

In Lawvere's proposal for a radically synthetic foundation for SDG \cite{MR3288694}, a object $T$ with a unique point $0:1\to T$ is proposed.  The connectedness of the pullback $R$ along $0$ of the evaluation $\ev_T^0:T^T\to T$  (i.e. $\Pi(R)=1$) is argued to be equivalent to 
 \begin{equation}\label{E:RretractSDG}\Pi(X^T)\cong\Pi(X)
 \end{equation}
for an arbitrary object $X$ under the assumption that $R$ be a retract of $T^T$ (see also \citet{MR4749715}).  \citet{MR4922622} observed that, in McLarty toposes, for a given $T$ (not necessarily atomic or with a unique point), \eqref{E:RretractSDG} holds if and only if $T$ is contractible---i.e.,  $T^X$ is connected for any object $X$.  Requiring atomicity would be reasonable as suggested by the following result. 

\begin{teoE} \hypertarget{T:TeorE} Every atomic object is contractible in any Grothendieck topos $\E$ for which the canonical geometric morphism is precohesive over $\set$. 

In a general McLarty topos, an atomic $T$ is contractible if and only if $2_!:2_T\to 2$ is an isomorphism.  
\end{teoE}
At the time of writing, the authors are unaware of noncontractible atomic objects in any precohesive context. 
\begin{que} Are there noncontractible atomic objects in McLarty toposes?
\end{que}
Under the weaker assumption of the presence of a connectedness structure in a Cartesian closed category (see \cite{MR4922622}) there ought to be more room for counterexamples, yet this goes beyond the scope of this manuscript.

The results presented herein exhibit how different the existence of right adjoints to exponential functors is from the intuitive notion of tiny.  There are examples of disconnected objects with right adjoints to their exponential functors, in the sense of the number of complemented subobjects (e.g. the terminal object in $\set\times\set$), and not every {\em indivisible}---i.e. connected---object is such that its exponential functor has a right adjoint. Those objects with right adjoints to their exponential functors that arise from axiomatics for Synthetic Differential Geometry are labeled as {\em infinitesimal} because their definition formally satisfies Leibniz's intuition, yet they are not {\em a priori} obtained by a geometrical limit (Appendix 4 of \cite{MR1083355} shows how some of the infinitesimals considered therein are, in some {\em a posteriori }sense, smaller than every open neighborhood of their unique point).  Therefore, neither atomic nor tiny nor infinitesimal are fully adequate names for such objects. Further knowledge of their behavior is needed to propose a better name.

The structure of this report is as follows. Section \ref{S:catAtoms} analyzes the full subcategory of atomic objects in Cartesian closed categories. Section \ref{S:retractatoms} studies the atomicity of retracts of atomic objects. Section \ref{S:jmap} begins the study of atomic objects in toposes, while Sections \ref{S:precohesiveAtom} and \ref{S:McLartyAtoms} specify to the precohesive context, first in general and next for McLarty toposes. Section \ref{S:Grothen} analyzes the special case of presheaf toposes and of Grothendieck toposes. Lastly, the promised examples and counterexamples are explained in Section \ref{S:Examples}.

\

{\bf Acknowledgements} The authors wish to thank Sebasti\'an Mart\'inez Ruiz and Omar Antol\'in for their suggestions. 
\section{The full subcategory of atomic objects}\label{S:catAtoms}
An atomic object in a cartesian closed category is an object $T$ such that the exponential functor $(-)^T$ has a right adjoint
\begin{equation}
(-)^T\dashv (-)_T.
\end{equation}
Denote by $\eta_T$ and $\varepsilon_T$ the unit and counit of this adjunction. Whenever there is no risk of confusion the reference to $T$ might be omitted.  Terminal objects are atomic, but initial objects cannot be atomic. Finite products of atomic objects are also atomic.  In general, limits of atomic objects are not atomic. In Section \ref{S:Grothen}, explicit counterexamples are provided. 

As in the case of the adjunction $(-)\times T\dashv (-)^T$, where natural transformations $(-)^f$ are provided given an arrow $f:S\to T$, for atomic objects $T$ and $S$ and an arrow $f:S\to T$, a natural transformation $(-)_f$ is determined  by the following diagram:
\begin{equation}\label{E:xtothefraction1overf}
\vcenter{\xymatrix{
(X_S)^T\ar[r]^(.53){(X_S)^f}\ar[d]_{(X_f)^T} & (X_S)^S\ar[d]^{\varepsilon_{S,X}} \\
(X_T)^T\ar[r]_(.6){\varepsilon_{T,X}} & X.
}}\end{equation}

For a fixed $X$, $T\mapsto X_T$ and $f\mapsto X_f$ is a covariant functor from the subcategory of atomic objects into the ambient category. 

\begin{lema}\label{L:superTnsubTarefaithful}
Let $\E$ be a cartesian closed category. If $T$ is an atomic object of $\E$ and $T$ has a point $0:1\rightarrow T$, then $(-)^T$ and $(-)_T$ are faithful.
\end{lema}
\begin{proof}
By \eqref{E:xtothefraction1overf}, the following diagram commutes:
\[\xymatrix{
(X_T)^1\ar[r]^(.53){(X_T)^{!}}\ar[d]_{(X_{!})^1} & (X_T)^T\ar[d]^{\varepsilon_{T,X}} \\
(X_1)^1\ar[r]_(.6){\varepsilon_{1,X}} & X.
}\]
On the other hand, since $\xymatrix{1\ar[r]^0 & T\ar[r]^{!} & 1\ar@{}|{=}[r] & 1_1}$,
\[
X_!X_0=1,
\]
and accordingly $X_!$ is epic. Therefore, since $\varepsilon_{1,X}$ is an isomorphism, $\varepsilon_{T,X}$ is epic, which means that $(-)_T$ is faithful.

Finally, $(-)^T$ is faithful for similar reasons: since $X^0\circ X^!=1$,  $X^0$ is epic and the corresponding diagram of definition for $(-)^0$ makes the counit of $(-)\times T\dashv (-)^T$ epic as well. 
\end{proof}
The diagram \eqref{E:xtothefraction1overf} can be generalized to natural transformations as follows.

\begin{teorema}
Let $\E$ be a cartesian closed category with atomic objects $S,T$. If $\varphi:(-)_T\Rightarrow(-)_S$ is a natural transformation, it determines a natural transformation $\psi:(-)^S\Rightarrow(-)^T$ as follows:
\begin{equation}\label{D:phiinducespsi}
\vcenter{\xymatrix{
X\ar[r]^(.45){\eta_{T,X}}\ar[d]_{\eta_{S,X}} & (X^T)_T\ar[d]^{\varphi_{X^T}} \\
(X^S)_S\ar[r]_{(\psi_X)_S} & (X^T)_S
}}\end{equation}
And vice versa, a natural transformation $\psi:(-)^S\Rightarrow(-)^T$ determines a natural transformation $\varphi:(-)_T\Rightarrow(-)_S$ as follows:
\begin{equation}\label{D:psiinducesphi}
\vcenter{\xymatrix{
(X_T)^S\ar[r]^{(\varphi_X)^S}\ar[d]_{\psi_{X_T}} & (X_S)^S\ar[d]^{\varepsilon_{S,X}} \\
(X_T)^T\ar[r]_(.55){\varepsilon_{T,X}} & X
}}\end{equation}
This is a bijective correspondence.
\end{teorema}
\begin{proof}
Let $\varphi:(-)_T\Rightarrow(-)_S$ be a natural transformation. Let $\psi_X:X^S\rightarrow X^T$ be the transpose of $\varphi_{X^T}\circ\eta_{T,X}$ under $(-)^S\dashv(-)_S$ as shown in \eqref{D:phiinducespsi}. So
\[
\psi_X=\varepsilon_{S,X^T}\circ(\varphi_{X^T}\circ\eta_{T,X})^S.
\]
Hence $\psi$ is natural in $X$. This $\psi_X$ makes \eqref{D:psiinducesphi} commute: indeed, the following diagram commutes:
\[\xymatrix{
(X_T)^S\ar[r]^(.45){(\eta_{T,X_T})^S}\ar[dr]_1 & (((X_T)^T)_T)^S\ar[r]^(.55){(\varphi_{(X_T)^T})^S}\ar[d]^{((\varepsilon_{T,X})_T)^S} & (((X_T)^T)_S)^S\ar[r]^(.62){\varepsilon_{S,(X_T)^T}}\ar[d]^{((\varepsilon_{T,X})_S)^S} & (X_T)^T\ar[d]^{\varepsilon_{T,X}} \\
& (X_T)^S\ar[r]_{(\varphi_X)^S} & (X_S)^S\ar[r]_{\varepsilon_{S,X}} & X
}\]

Conversely, let $\varphi_X:X_T\rightarrow X_S$ be the transpose of $\varepsilon_{T,X}\circ\psi_X$ under $(-)^S\dashv(-)_S$ as shown in \eqref{D:psiinducesphi}. So
\[
\varphi_X=(\varepsilon_{T,X}\circ\psi_{X_T})_S\circ\eta_{S,X_T}.
\]
Hence $\varphi$ is natural in $X$. This $\phi_X$ makes \eqref{D:phiinducespsi} commute: indeed, the following diagram commutes:
\[\xymatrix{
X\ar[r]^{\eta_{S,X}}\ar[d]_{\eta_{T,X}} & (X^S)_S\ar[r]^{(\psi_X)_S}\ar[d]_{((\eta_{T,X})^S)_S} & (X^T)_S\ar[d]_{((\eta_{T,X})^T)_S}\ar[dr]^1 & \\
(X^T)_T\ar[r]_(.45){\eta_{S,(X^T)_T}} & (((X^T)_T)^S)_S\ar[r]_(.53){(\psi_{(X^T)_T})_S} & (((X^T)_T)^T)_S\ar[r]_(.61){(\varepsilon_{T,X^T})_S} & (X^T)_S
}\]
Therefore there is a bijective correspondence between natural transformations $\varphi:(-)_T\Rightarrow(-)_S$ and natural transformations $\psi:(-)^S\Rightarrow(-)^T$.
\end{proof}
In fact, this is also true of any given adjunction between endofunctors. In particular, one has the following result.  
\begin{teorema}
Let $\E$ be a cartesian closed category with objects $S,T$. There is a bijective correspondence between natural transformations $\varphi:(-)^S\Rightarrow(-)^T$ and natural transformations $\psi:(-)\times T\Rightarrow(-)\times S$.
\end{teorema}
\begin{proof}
Similar to that of the previous theorem.
\end{proof}
\begin{prop}
Let $\E$ be a cartesian closed category and $\varphi$ be a natural transformation $(-)^A\Rightarrow(-)^B$. Let $\psi:(-)\times B\Rightarrow(-)\times A$ be the natural transformation induced by $\varphi$. If  for $\varphi_X$ the diagram 
\begin{equation}\label{D:sigmacommutesphi}
\vcenter{\xymatrix{
X\ar[r]^{\sigma_X^A}\ar[dr]_{\sigma_X^B} & X^A\ar[d]^{\varphi_X} \\
& X^B,
}}\end{equation}
where $\sigma_X^A$ is the transpose to the projection under the adjuntion $(-)\times A\dashv (-)^A$, then so does the following diagram:
\[\xymatrix{
X\times B\ar[dr]^{\pi_{B,X}}\ar[d]_{\psi_X} & \\
X\times A\ar[r]_(.6){\pi_{A,X}} & X.
}\]
\end{prop}
\begin{proof}
The equivalence is given by the following commutative diagrams:
\begin{equation}\label{D:phiinducedbypsithroughev}
\vcenter{\xymatrix{
X^A\times B\ar[r]^{\varphi_X\times B}\ar[d]_{\psi_{X^A}} & X^B\times B\ar[d]^{\ev_{B,X}} && X\ar[r]^(.45){\eta_{A,X}}\ar[d]_{\eta_{B,X}} & (X\times A)^A\ar[d]^{\varphi_{X\times A}}\\
X^A\times A\ar[r]_(.6){\ev_{A,X}} & X && (X\times B)^B\ar[r]_{(\psi_X)^B} & (X\times A)^B
}}\end{equation}
Now, by the naturality of $\psi$ and \eqref{D:phiinducedbypsithroughev}, the following diagram commutes:
\[\xymatrix{
& &&X\times B\ar@/_2pc/[dlll]^{\psi_X}\ar[dll]_{\sigma_X^A\times B}\ar[d]^{\sigma_X^B\times B}\ar@/^4pc/[dd]^{\pi_{B,X}}  \\
X\times A\ar[dr]^{\sigma_X^A\times A}\ar@/_4pc/[drrr]_{\pi_{A,X}} & X^A\times B\ar[rr]_{\varphi_X\times B}\ar[d]^{\psi_{X^A}}& & X^B\times B\ar[d]_{\ev_{B,X}} \\
 & X^A\times A\ar[rr]_(.6){\ev_{A,X}} & & X
}\]
\end{proof}

\begin{teorema}
Let $\E$ be a category with finite products, $A,B\in\E$, and $\psi:(-)\times B\Rightarrow(-)\times A$ be a natural transformation. Then $\psi=(-)\times f$ for some arrow $f:B\rightarrow A$ iff the following diagram commutes:
\[\xymatrix{
X\times B\ar[dr]^{\pi_{B,X}}\ar[d]_{\psi_X} & \\
X\times A\ar[r]_(.6){\pi_{A,X}} & X
}\]
\end{teorema}
\begin{proof}
Let $f:=\pi_A\circ\psi_1\circ\pi_B^{-1}$. Since the following diagram commutes:
\[\xymatrix{
& & B\ar[dr]^1 & \\
X\ar[d]_1 & X\times B\ar[l]_{\pi_X}\ar[d]^{\psi_X}\ar[r]_{!\times B}\ar[ur]^{\pi_B} & 1\times B\ar[r]_{\pi_B}\ar[d]^{\psi_1} & B\ar[d]^f \\
X & X\times A\ar[l]^{\pi_X}\ar[r]^{!\times A}\ar[dr]_{\pi_A} & 1\times A\ar[r]^{\pi_A} & A, \\
& & A\ar[ur]_1 &
}\]
we have $\psi_X=X\times f$.
\end{proof}
Since $\pi$ is isomorphic to $(-)\times!$ and $\sigma$ is isomorphic to $(-)^!$, a similar relation is obtained between commutation with $(-)^!$ and commutation with $(-)_!$, as summarized in the following result.
\begin{cor}\label{C:Conditionforfullness}
Let $\E$ be a cartesian closed category with atomic objects $S,T$. Let $\varphi:(-)_T\Rightarrow(-)_S$ be a natural transformation. Then $\varphi=(-)_f$ for some $f:T\rightarrow S$ iff the following diagram commutes:
\[\xymatrix{
X_T\ar[r]^(.53){\varphi_X}\ar[dr]_{X_{!_T}} & X_S\ar[d]^{X_{!_S}} \\
& X.
}\]
\end{cor}
\begin{proof} Straightforward. 
\end{proof}

\section{Retracts of atomic objects}\label{S:retractatoms}
The purpose of this section is to prove the first statement of \hyperlink{T:TeorB}{Theorem B}. Namely,
\begin{teorema}\label{T:retractoftinytiny}
Let $\E$ be a regular cartesian closed category. Any retract of an atomic object of $\E$ is atomic.
\end{teorema}
\begin{proof}
Let $T$ be an atomic object of $\E$. Let $Q$ be a retract of $T$; that is, there are arrows $q:T\twoheadrightarrow Q$ and $i:Q\rightarrowtail T$ such that $q\circ i=1_Q$. Let $r:=i\circ q:T\rightarrow T$. For every $Y\in\E$, let
\[\xymatrix{
Y_T\ar@{->>}[dr]_{\vartheta}\ar[rr]^{Y_r} & & Y_T \\
& Z\ar@{(>->}[ur]_{\varphi} &
}\]
be a chosen epi-mono factorization of $Y_r$.  So define
\[
\Psi f:=\vartheta\circ\widehat{f\circ X^i}.
\]
for $f:X^Q\rightarrow Y$ in $\E$, and
\[
\Psi^{-1} g:=\widehat{\varphi\circ g}\circ X^q
\]
for $g:X\rightarrow Z$ in $\E$, where $\widehat{(-)}$ is the transpose under the adjunction $(-)^T\dashv(-)_T$. The assignments $\Psi$ and $\Psi^{-1}$ are inverse of each other. Indeed, since the following diagram commutes:
\[\xymatrix{
X^T\ar[r]^{X^r}\ar[d]_{(\widehat{f\circ X^i})^T} & X^T\ar[d]_(.42){(\widehat{f\circ X^i})^T}\ar@/^2pc/[dd]^{f\circ X^i} \\
(Y_T)^T\ar[r]^(.53){(Y_T)^r}\ar[d]_{(Y_r)^T} & (Y_T)^T\ar[d]_{\varepsilon_{T,Y}} \\
(Y_T)^T\ar[r]_(.6){\varepsilon_{T,Y}} & Y,
}\]
the transpose of $Y_r\circ\widehat{f\circ X^i}$ is $f\circ X^i\circ X^r$, which is $f\circ X^i$. Hence
\[
\Psi^{-1}\Psi(f)=f\circ X^i\circ X^q=f.
\]

Conversely, one has $\widehat{\varphi\circ g}\circ X^q\circ X^i=\widehat{\varphi\circ g}\circ X^r$. A commutative diagram similar to the previous one yields that the transpose of $\widehat{\varphi\circ g}\circ X^r$ is $Y^r\circ\varphi\circ g$, which is
\[
\varphi\circ\vartheta\circ\varphi\circ g=\varphi\circ g,
\]
since $Y_r$ is idempotent, and $\vartheta$ and $\varphi$ are epic and monic, respectively. Hence
\[
\Psi\Psi^{-1}(g)=\vartheta\circ\varphi\circ g=g.
\]

To see that $\Psi$ is natural both in $X$ and $Y$, let $h:X'\rightarrow X$ be an arrow in $\E$. The following diagrams commute:
\[\xymatrix{
(X')^T\ar[r]^{h^T}\ar[d]_{(X')^i} & X^T\ar[d]^{X^i} \\
(X')^Q\ar[r]_{h^Q} & X^Q 
}\]
and
\[\xymatrix{
(X')^T\ar[d]_{h^T} & \\
X^T\ar[dr]^{f\circ X^i}\ar[d]_{(\widehat{f\circ X^i})^T} & \\
(Y_T)^T\ar[r]_(.62){\varepsilon_{T,Y}} & Y.
}\]
Hence
\begin{align*}
\widehat{fh^Q(X')^i} &=\widehat{fX^ih^T}\\
&=\widehat{fX^i}h.
\end{align*}
Hence $\Psi$ is natural in $X$; i.e., the following diagram commutes:
\[\xymatrix{
\E(X^Q,Y)\ar[r]^\Psi\ar[d]_{\E(h^Q,Y)} & \E(X,Z)\ar[d]^{\E(h,Z)}\\
\E((X')^Q,Y)\ar[r]_\Psi & \E(X',Z).
}\]

Now, let $k:Y\rightarrow Y'$ be an arrow in $\E$. Consider the following commutative diagram:
\[\xymatrix{
Y_T\ar[rr]^{Y_r}\ar@{->>}[dr]^\vartheta\ar[dd]_{k_T} & & Y_T\ar[dd]^{k_T} \\
& Z\ar@{(>->}[ur]^\varphi & \\
(Y')_T\ar[rr]_{(Y')_r}\ar@{->>}[dr]_{\vartheta'} & & (Y')_T \\
& Z'\ar@{(>->}[ur]_{\varphi'} &  
}\]
The arrow $\vartheta'\circ k_T\circ\varphi:Z\rightarrow Z'$ makes the above diagram commute. So the following diagram commutes:
\[\xymatrix{
\E(X^Q,Y)\ar[r]^\Psi\ar[d]_{\E(X^Q,k)} & \E(X,Z)\ar[d]^{\E(X,\vartheta'\circ k_T\circ\varphi)} \\
\E(X^Q,Y')\ar[r]_\Psi & \E(X,Z'),
}\]
since
\begin{align*}
\vartheta'\circ k_T\circ\varphi\circ\vartheta\circ\widehat{f\circ X^i} &=\vartheta'\circ k_T\circ Y_r\circ\widehat{f\circ X^i}\\
&=\vartheta'\circ(Y')_r\circ k_T\circ\widehat{f\circ X^i}\\
&=\vartheta'\circ\widehat{k f X^i}.
\end{align*}
Whence $\Psi$ is natural in $Y$. Therefore $(-)^Q$ has a right adjoint functor.
\end{proof}

\section{A generalized singleton in elementary toposes}\label{S:jmap}

The purpose of this section is to prove \hyperlink{T:TeorA}{Theorem A}. This is done in several parts. 
\begin{teorema}\label{T:Monicity of j}
Let $\E$ be a topos and $T$ an atomic object with a point $p:1\rightarrow T$ of $T$. Then for every $Y\in\E$ the transpose $j_{p,Y}$ of $(\delta_Y)_T\circ 1\times Y_p$ is monic. 
\end{teorema}
\begin{teorema}\label{T:relationbetweensingletonandj}
Let $\E$ be a topos and be $T\in\E$ an atomic object with a point $p:1\rightarrow T$. The following diagram commutes:
\[\xymatrix{
Y\ar[r]^{\{-\}_Y}\ar[d]_{Y_p} & \subob^Y\ar[d]^{(\subob_p)^Y} \\
Y_T\ar[r]_{j_{p,Y}} & (\subob_T)^Y
}\]
\end{teorema}
\begin{teorema}\label{T:functionsnj}
Let $\E$ be a topos and $T$ be an atomic object of $\E$ with a point $p:1\rightarrow T$. If $S$ is another atomic object of $\E$ and $f:T\rightarrow S$ is an arrow in $\E$, the following diagram commutes:
\begin{equation}\label{E:Subarrow and j}\xymatrix{
X_T\ar[r]^{X_f}\ar@{(>->}[d]_{j_{p,X}} & X_S\ar@{(>->}[d]^{j_{f\circ p,X}} \\
(\subob_T)^X\ar[r]_{(\subob_f)^X} & (\subob_S)^X.
}\end{equation}
(For more precision in notation since needed, a dependency is added, as a subindex,  on the point related to $j_X$.)
\end{teorema}
\begin{cor}
Let $\E$ be a topos and $T$ be an atomic object of $\E$ with a point $p:1\rightarrow T$. Then $j_{!\circ p,X}=\{-\}_X$.
\end{cor}
\begin{proof}
By \ref{T:relationbetweensingletonandj} and \ref{T:functionsnj}, the following diagram commutes:
\[\xymatrix{
X\ar[r]_{X_p}\ar[d]_{\{-\}_X}\ar@/^1.2pc/[rr]^1 & X_T\ar[r]_{X_{!}}\ar[d]_{j_{p,X}} & X\ar[d]^{j_{!\circ p,X}} \\
\subob^X\ar[r]^(.45){(\subob_p)^X}\ar@/_1.2pc/[rr]_1 & (\subob_T)^X\ar[r]^(.55){(\subob_{!})^X} & \subob^X
}\]
\end{proof}
\begin{teorema}\label{T:CounitOmega}
Let $\E$ be a topos and $T\in\E$ an atomic object with a point $p:1\rightarrow T$. The following diagram commutes:
\[\xymatrix{
(X_T)^T\ar[r]^{(j_{p,X})^T}\ar[d]_{\epsilon_X} & ((\subob_T)^X)^T\ar[r]^{\alpha_{\subob_T}} & ((\subob_T)^T)^X\ar[d]^{(\epsilon_\subob)^X} \\
X\ar[rr]_{\{-\}_X} & & \subob^X,
}\]
where $\alpha:((-)^T)^X\Rightarrow((-)^X)^T$ is the exponent swapping natural transformation.
\end{teorema}
The following lemma is required to prove the monicity of $j_{p,X}$.
\begin{lema}\label{L:commutativediagwithYp}
Let $\E$ be a topos and $T$ be a atomic object with a point $p:1\rightarrow T$. The following diagram commutes:
\[\xymatrix{
Y_T\times Y\ar[r]^(.45){(\widehat{\delta_Y})_T\times 1}\ar[d]_{1\times Y_p} & (\subob^Y)_T\times Y\ar[r]^{(\subob^{Y^p})_T\times 1} & (\subob^{Y^T})_T\times Y\ar[d]^{\eta_{T,(\subob^{Y^T})_T\times Y}} \\
Y_T\times Y_T\ar[d]_\xi & & (((\subob^{Y^T})_T\times Y)^T)_T\ar[d]^{\beta_T} \\
(Y\times Y)_T\ar[ddrr]_{(\delta_Y)_T} & & (((\subob^{Y^T})_T)^T\times Y^T)_T\ar[d]^{(\varepsilon_{T,\subob^{Y^T}}\times 1)_T} \\
& & (\subob^{Y^T}\times Y^T)_T\ar[d]^{\ev_T} \\
& & \subob_T
}\]

\end{lema}
\begin{proof}
The following diagrams commute:
\[\xymatrix{
Y\ar[r]^(.4){\eta_{T,Y}}\ar[dr]_{Y_p} & (Y^T)_T\ar[d]^{(Y^p)_T} \\
& Y_T
}\]
\[\xymatrix{
\subob^Y\times Y^T\ar[r]^(.53){1\times Y^p}\ar[d]_{\subob^{Y^p}\times 1} & \subob^Y\times Y\ar[d]^\ev \\
\subob^{Y^T}\times Y^T\ar[r]_(.6)\ev & \subob
}\]
\[\xymatrix{
Y\times Y\ar[dr]^{\delta_Y}\ar[d]_{\widehat{\delta_Y}\times 1} & \\
\subob^Y\times Y\ar[r]_(.6)\ev & \subob
}\]
Since $\eta_T:1\Rightarrow((-)^T)_T$ is a natural transformation and $((-)^T)_T$ preserves products, the following diagram commutes:
\[\xymatrix{
(\subob^{Y^T})_T\times Y\ar[rrr]^(.42){\eta_{T,(\subob^{Y^T})_T}\times\eta_{T,Y}}\ar[drrr]_{\eta_{T,(\subob^{Y^T})_T\times Y}} & & & (((\subob^{Y^T})_T)^T)_T\times(Y^T)_T\ar@{}@<-.5ex>[d]^[right]{\cong} \\
& & & (((\subob^{Y^T})_T\times Y)^T)_T
}\]
Hence, by the commutativity of the previous diagrams, the following diagram commutes:
\[\xymatrix{
Y_T\times Y\ar[r]^(.45){(\widehat{\delta_Y})_T\times 1}\ar[d]_{1\times\eta_{T,Y}} & (\subob^Y)_T\times Y\ar[r]^(.47){(\subob^{Y^p})_T\times 1}\ar[d]_{1\times\eta_{T,Y}} & (\subob^{Y^T})_T\times Y\ar[d]^{1\times\eta_{T,Y}} \\
Y_T\times (Y^T)_T\ar[r]^(.43){(\widehat{\delta_Y})_T\times 1}\ar[d]_{1\times(Y^p)_T} & (\subob^Y)_T\times (Y^T)_T\ar[r]^(.48){(\subob^{Y^p})_T\times 1}\ar[d]_{1\times(Y^p)_T} & (\subob^{Y^T}\times(Y^T)_T\ar@/^1.5pc/[ddl]^{\ev_T} \\
Y_T\times Y_T\ar@/_1pc/[dr]_{(\delta_Y)_T}\ar[r]^(.45){(\widehat{\delta_Y})_T\times 1} & (\subob^Y\times Y)_T\ar[d]_{\ev_T} & \\
& \subob_T & 
}\]
\end{proof}

\begin{proof}[Proof of \ref{T:Monicity of j}]
Let $\vartheta$ denote the composite
\[\xymatrix{
Y_T\ar[r]^(.43){(\widehat{\delta_Y})_T} & (\subob^Y)_T\ar[r]^(.45){(\subob^{Y^p})_T} & (\subob^{Y^T})_T.
}\]
Now, consider the following sequence of bijections natural in $X$:
\[\begin{tabular}{c}
$X\rightarrow(\subob^{Y^T})_T$ \\
\hline
$X^T\rightarrow\subob^{Y^T}$ \\
\hline
$X^T\times Y^T\rightarrow\subob$ \\
\hline
$(X\times Y)^T\rightarrow\subob$ \\
\hline
$X\times Y\rightarrow\subob_T$ \\
\hline
$X\rightarrow(\subob_T)^Y$.
\end{tabular}\]
Therefore $(\subob^{Y^T})_T\cong(\subob_T)^Y$. Let $\varphi$ denote that isomorphism. From the sequence of natural bijections above, it is possible to calculate the transpose $\hat{\varphi}:(\subob^{Y^T})_T\times Y\rightarrow\subob_T$ of $\varphi$: it is precisely the composite arrow from $(\subob^{Y^T})_T\times Y$ to $\subob_T$ in the diagram in \ref{L:commutativediagwithYp}.

Since $Y^pY^!=1$, $Y^p$ is epic. Hence $\subob^{Y^p}$ is monic. Clearly $\widehat{\delta_Y}$ is monic since it is $\{-\}$. Now, as $(-)_T$ is right adjoint, it preserves monics. Hence $\varphi\vartheta$ is monic, whose transpose is $\hat{\varphi}\circ\vartheta\times 1=(\delta_Y)_T\circ 1\times Y_p$ by \ref{L:commutativediagwithYp}.
\end{proof}

\begin{proof}[Proof of \ref{T:relationbetweensingletonandj}]
The diagram commutes if the transposes of the different paths in the diagram commute. The following commtative diagrams yield the respective transposes:
\[\xymatrix{
Y\times Y\ar[d]_{\widehat{\delta_y}\times 1}\ar[dr]^{\delta_Y} & \\
\subob^Y\times Y\ar[r]_(.6)\ev\ar[d]_{(\subob_p)^Y\times 1} & \subob\ar[d]^{\subob_p} \\
(\subob_T)^Y\times Y\ar[r]_(.62)\ev & \subob_T
}\]
and
\[\xymatrix{
Y\times Y\ar[d]_{Y_p\times 1} & \\
Y_T\times Y\ar[d]_{j_{p,Y}\times 1}\ar[dr]^{(\delta_Y)_T\circ 1\times Y_p} & \\
(\subob_T)^Y\times Y\ar[r]_(.6)\ev & \subob_T.
}\]
That is, the following diagram should commute:
\[\xymatrix{
Y\times Y\ar[rr]^{\delta_Y}\ar[d]_{Y_p\times 1} & & \subob\ar[d]^{\subob_p} \\
Y_T\times Y\ar[r]_{1\times Y_p} & Y_T\times Y_T\ar[r]_(.6){(\delta_Y)_T} & \subob_T
}\]
But it clearly commutes since the diagram
\[\xymatrix{
Y\times Y\ar[r]^(.6){\delta_Y}\ar[d]_{(Y\times Y)_p} & \subob\ar[d]^{\subob_p} \\
(Y\times Y)_T\ar[r]_(.62){(\delta_Y)_T} & \subob_T
}\]
commutes.
\end{proof}

\begin{proof}[Proof of \ref{T:functionsnj}]
Since $\subob_f\circ(\{-\}_X)_T=(\{-\}_X)_S\circ X_f$,  internally the diagram \eqref{E:Subarrow and j} commutes:
\[\xymatrix{
\xi\ar@{|->}[d]\ar@{|->}[rr] & & X_f(\xi)\ar@{|->}[d] \\
(x\mapsto\xi=_TX_p(x))\ar@<-.7ex>@{|->}[rr] & & *\txt{$(x\mapsto X_f(\xi)=_SX_{f\circ p}(x))$\\$(x\mapsto\subob_f(\xi=_TX_p(x))$}
}\]
\end{proof}

\begin{lema}\label{L:distdistswapequalswapdistdist}
Let $\E$ be a topos. If $X,T,A,B\in\E$ and $\alpha$ is the induced natural isomorphism $((-)^X)^T\Rightarrow((-)^T)^X$ and $\beta_{A,B}$ is the induced natural isomorphism $(A\times B)^X\cong A^X\times B^X$ for any $A,B\in\E$, then the following diagram commutes:
\[\xymatrix{
((A\times B)^X)^T\ar[r]^{\beta^T}\ar[d]_{\alpha_{A\times B}} & (A^X\times B^X)^T\ar[r]^\beta & (A^X)^T\times (B^X)^T\ar[d]^{\alpha_A\times\alpha_B} \\
((A\times B)^T)^X\ar[r]_{\beta^X} & (A^T\times B^T)^X\ar[r]_\beta & (A^T)^X\times (B^T)^X,
}\]
\end{lema}
\begin{proof}
By using the internal language of $\E$, if $u\in((A\times B)^X)^T$, one has
\[
u=(t\mapsto(x\mapsto\langle u_1(t)(x),u_2(t)(x)\rangle)).
\]
Hence, if applying $\beta^T$ to it, one gets
\[
(t\mapsto\langle(x\mapsto u_1(t)(x)),(x\mapsto u_2(t)(x))\rangle).
\]
Applying $\beta$ in turn, one gets
\[
\langle(t\mapsto(x\mapsto u_1(t)(x))),(t\mapsto(x\mapsto u_2(t)(x)))\rangle.
\]
Finally  swapping exponents, it follows that
\[
\langle(x\mapsto(t\mapsto u_1(t)(x))),(x\mapsto(t\mapsto u_2(t)(x)))\rangle.
\]

On the other hand, applying $\alpha_{A\times B}$ to $u$, one gets
\[
(x\mapsto(t\mapsto\langle u_1(t)(x),u_2(t)(x)\rangle)).
\]
It is clear that continuing with the process established by the other path in the diagram, one arrives at the same result.
\end{proof}

\begin{lema}\label{L:swappingwithsigmannames}
Let $\E$ be a topos. If $A,X,T\in\E$ and $\alpha:((-)^X)^T\Rightarrow((-)^T)^X$ is the exponent swapping natural transformation then
\[
\alpha_A\circ(\sigma_A^X)^T=\sigma_{A^T}^X\qquad\text{and}\qquad\alpha_X\circ(\text{`$1_X$'})^T=\text{`$\sigma_X^T$'}.
\]
\end{lema}
\begin{proof}
By using the internal language, one gets
\[
(\sigma_A^X)^T=(u\mapsto\sigma_A^X\circ u)\quad\text{and}\quad\sigma_{A^T}^X=(u\mapsto(x\mapsto u)).
\]
On the other hand,
\[
\sigma_A^X\circ u =(a\mapsto(x\mapsto a))\circ(t\mapsto u(t))=(t\mapsto(x\mapsto u(t)).
\]
Hence, by swapping exponents, one gets
\[
(x\mapsto(t\mapsto u(t))=(x\mapsto u).
\]
Therefore $\alpha_A\circ(\sigma_A^X)^T=\sigma_{A^T}^X$.

Now, internally,
\[
(\text{`$1_X$'})^T=(v\mapsto\text{`$1_X$'}\circ v).
\]
Hence
\[
(v\mapsto\text{`$1_X$'}\circ v)= (\ast\mapsto(x\mapsto x))\circ(t\mapsto\ast)=(t\mapsto(x\mapsto x)).
\]
By swapping exponents, one gets
\[
(x\mapsto(t\mapsto x)).
\]
That is, one gets $(v\mapsto\sigma_X^T):1^T\rightarrow(X^T)^X$, which is `$\sigma_X^T$'.
\end{proof}
\begin{proof}[Proof of \ref{T:CounitOmega}]
Let $\eta$ and $\varepsilon$ be the unit and counit of the adjunction $(-)^T\dashv(-)_T$, resp. Let $\vartheta$ and $\ev$ be the unit and counit of the adjuntion $(-)\times T\dashv(-)^T$. Let $\alpha$ be the induced natural isomorphism $((-)^X)^T\Rightarrow((-)^T)^X$ for $X\in\E$ and let $\beta_{A,B}$ be the induced natural isomorphism $(A\times B)^X\cong A^X\times B^X$ for any $A,B\in\E$. (We drop the subindices whe no risk of confusion). Denote the composite $\ev\circ\varepsilon\times 1\circ\beta$ by $\varphi$. The following outer diagram commutes:
\[\xymatrix{
(X_T)^T\ar[d]_{((\{-\})_T)^T}\ar[rr]^{\varepsilon_X}\ar@{}[ddrrr]|{(1)} & & X\ar[r]^{\{-\}} & \subob^X\ar[d]^{\subob^{X^p}} \\
((\subob^X)_T)^T\ar[d]_{((\subob^{X^p})_T)^T} & & & \subob^{X^T}\ar[d]^{\langle\sigma,\text{`$\sigma$'}!\rangle}\\
((\subob^{X^T})_T)^T\ar[dd]_{\left(\vartheta_{(\subob^{X^T})_T}\right)^T}\ar[r] & \bullet\ar[r]^{\alpha\times\alpha} & \bullet\ar[r]^(.3){\varepsilon^X\times 1} & (\subob^{X^T})^X\times(X^T)^X\ar[dd]^{\beta^{-1}} \\
& & & \\
(((\subob^{X^T})_T\times X)^X)^T\ar[d]_{\left(\left(\eta_{(\subob^{X^T})_T\times X}\right)^X\right)^T}\ar@/^.4pc/[uur]\ar@/_1pc/[uurr]\ar@<4ex>@{}[urr]|{(3)}\ar@<4ex>@{}[uur]|{(2)} & & & (\subob^{X^T}\times X^T)^X\ar[dd]^{\ev^X}\\
(((((\subob^{X^T})_T\times X)^T)_T)^X)^T\ar[d]_{((\varphi_T)^X)^T} & & & \\
((\subob_T)^X)^T\ar[rr]_{\alpha_{\subob_T}}\ar@{}[uuurrr]|{(4)} & & ((\subob_T)^T)^X\ar[r]_(.6){(\varepsilon_\subob)^X} & \subob^X
}\]
Indeed, Diagram (4) commutes:
\[\xymatrix{
(((\subob^{X^T})_T\times X)^X)^T\ar[rr]^{\left(\left(\eta_{(\subob^{X^T})_T\times X}\right)^X\right)^T}\ar[d]_\alpha & & (((((\subob^{X^T})_T\times X)^T)_T)^X)^T\ar[r]^(.7){((\varphi_T)^X)^T}\ar[d]_\alpha & ((\subob_T)^X)^T\ar[d]^{\alpha_{\subob_T}} \\
(((\subob^{X^T})_T\times X)^T)^X\ar[rr]^{\left(\left(\eta_{(\subob^{X^T})_T\times X}\right)^T\right)^X}\ar@/_1.3pc/[drr]_1\ar[d]_{\beta^X} & & (((((\subob^{X^T})_T\times X)^T)_T)^T)^X\ar[r]^(.7){((\varphi_T)^T)^X}\ar[d]_{\left(\varepsilon_{((\subob^{X^T})_T\times X)^T}\right)^X} & ((\subob_T)^T)^X\ar[d]^{(\varepsilon_\subob)^X} \\
(((\subob^{X^T})_T)^T\times X^T)^X\ar@/_.7pc/[drr]_{(\varepsilon\times 1)^X}\ar[d]_\beta & & (((\subob^{X^T})_T\times X)^T)^X\ar[r]_(.65){\varphi^X} & \subob^X \\
(((\subob^{X^T})_T)^T)^X\times (X^T)^X\ar@/_.7pc/[drr]_{\varepsilon^X\times 1} & & (\subob^{X^T}\times X^T)^X\ar@/_.7pc/[ur]_{\ev^X} & \\
& & (\subob^{X^T})^X\times(X^T)^X\ar[u]_{\beta^{-1}}; &
}\]
by \ref{L:distdistswapequalswapdistdist}, Diagram (3) commutes:
\[\xymatrix{
(((\subob^{X^T})_T\times X)^X)^T\ar[r]^{\beta^T}\ar[d]_\alpha & (((\subob^{X^T})_T)^X\times X^X)^T\ar[r]^\beta & (((\subob^{X^T})_T)^X)^T\times (X^X)^T\ar[d]^{\alpha\times\alpha} \\
(((\subob^{X^T})_T\times X)^T)^X\ar[r]_{\beta^X} & (((\subob^{X^T})_T)^T\times X^T)^X\ar[r]_\beta & (((\subob^{X^T})_T)^T)^X\times (X^T)^X;
}\]
Diagram (2) commutes:
\[\xymatrix{
((\subob^{X^T})_T)^T\ar[rr]^(.45){\left(\vartheta_{(\subob^{X^T})_T}\right)^T}\ar[drr]^(.55){\langle\sigma,\text{`$1$'}!\rangle^T}\ar@/_.7pc/[ddrr]_{\langle\sigma^T,(\text{`$1$'}!)^T\rangle} & & (((\subob^{X^T})_T\times X)^X)^T\ar[d]^{\beta^T} \\
& & (((\subob^{X^T})_T)^X\times X^X)^T\ar[d]^\beta \\
& & (((\subob^{X^T})_T)^X)^T\times (X^X)^T,
}\]
and, by \ref{L:swappingwithsigmannames} and the naturality of $\sigma$, Diagram (1) commutes:
\[\xymatrix{
(X_T)^T\ar[d]_{((\{-\})_T)^T}\ar[r]^{\varepsilon_X} & X\ar[r]^{\{-\}} & \subob^X\ar[dd]^{\subob^{X^p}} \\
((\subob^X)_T)^T\ar[d]_{((\subob^{X^p})_T)^T}\ar@/_.7pc/[urr]^\varepsilon & & \\
((\subob^{X^T})_T)^T\ar[d]^{\langle\sigma^T,(\text{`$1$'}!)^T\rangle}\ar@/^1.5pc/[ddrr]^(.6){\langle\varepsilon^X\sigma,\text{`$\sigma_X^T$'}!\rangle}\ar[rr]^(.53)\varepsilon\ar@/_6.5pc/[dd]_{\langle\sigma,\text{`$\sigma_X^T$'}!\rangle} & & \subob^{X^T}\ar[dd]^{\langle\sigma,\text{`$\sigma$'}!\rangle} \\
(((\subob^{X^T})_T)^X)^T\times (X^X)^T\ar[d]^{\alpha\times\alpha} & & \\
(((\subob^{X^T})_T)^T)^X\times (X^T)^X\ar[rr]_(.55){\varepsilon^X\times 1} & & (\subob^{X^T})^X\times(X^T)^X.
}\]

The following diagram commutes:
\[\xymatrix{
\subob^X\ar@/^.7pc/[ddrr]^1\ar[d]_{\subob^{X^p}} & & \\
\subob^{X^T}\ar[d]_{\left\langle\sigma_{\subob^{X^T}}^X,\text{`$\sigma_X^T$'}!\right\rangle}  & & \\
(\subob^{X^T})^X\times(X^T)^X\ar[r]_(.55){\beta^{-1}} & (\subob^{X^T}\times X^T)^X\ar[r]_(.7){\ev^X} & \subob^X.
}\]
Indeed, by using the internal language of the topos, each arrow in the diagram in the up-down-left-right path is described as follows:
\begin{align*}
&u\mapsto(k\mapsto u(k(p)))\\
&v\mapsto\langle x\mapsto v,x\mapsto(t\mapsto x)\rangle\\
&\langle x\mapsto w(x),x\mapsto z(x)\rangle\mapsto(x\mapsto\langle w(x),z(x)\rangle)\\
&f\mapsto(x\mapsto f_1(x)(f_2(x))).
\end{align*}
So by composing them, one gets
\[\xymatrix{
u\ar@{|->}[d]\\
k\mapsto u(k(p))\ar@{|->}[d]\\
\langle x\mapsto(k\mapsto u(k(p))),x\mapsto(t\mapsto x)\rangle\ar@{|->}[d]\\
x\mapsto\langle k\mapsto u(k(p)),t\mapsto x\rangle\ar@{|->}[d]\\
x\mapsto u(x)
}\]
since $t\mapsto x$ is substituted for $k$. That is, the diagram above commutes.

Finally, $j_{p,X}$ is the transpose of the $(\delta_X)_T\circ 1\times X_p$ under the adjunction $(-)\times X\dashv(-)^X$, or alternatively the transpose of the other composite in the diagram of \ref{L:commutativediagwithYp}. Let $g$ be the following subcomposite of that composite: from $\eta$ to $\ev^T$, and substitute $X$ for $Y$. That is, $g:(\subob^{X^T})_T\times X\rightarrow\subob_T$. So the other subcomposite is $(\subob^{X^p})_T\times 1\circ(\{-\})_T\times 1$. Hence, by naturality of $\vartheta$,
\begin{align*}
j_{p,X} &=g^X\circ((\subob^{X^p})_T\times 1\circ(\{-\})_T\times 1)^X\circ\vartheta_{X_T}\\
&= g^X\circ\vartheta_{(\subob^{X^T})_T}\circ\subob^{X^p}\circ(\{-\})_T.
\end{align*}
This finishes the proof.
\end{proof}

\section{General axiomatic cohesion} \label{S:precohesiveAtom}
The purpose of this section is to prove  \hyperlink{T:TeorC}{Theorem C}. This is split into two theorems. 
\begin{teorema}\label{T:ExpIdealsReflect}
Let $\E$ be a cartesian closed category, and $\mathcal{S}$ a reflective and coreflective subcategory of $\E$ with reflector $L$ and coreflector $R$ on $\E$. If $\mathcal{S}$ is an exponential ideal, then the inclusion funtor $\mathcal{S}\hookrightarrow\E$ reflects atomic objects.
\end{teorema}
\begin{proof}
Let $A,B,C\in\mathcal{S}$ and suppose $B$ is atomic in $\E$. Consider the following bijections natural in $A$ and $C$:
\[\begin{tabular}{c}
$A^B\rightarrow C$ \\
\hline
$A\rightarrow C_{B}$ \\
\hline
$A\rightarrow R(C_{B})$
\end{tabular}\]
\end{proof}

\begin{teorema}\label{T:fshriekpreservesatoms}

Let $F:\E\rightarrow\mathcal{S}$ be a precohesive geometric morphism. Then $F_!$ preserves atomic objects.
\end{teorema}

\begin{lema}
Let $\E$ be a cartesian closed category, and $\mathcal{S}$ a reflective and coreflective subcategory of $\E$ with reflector $L$ and coreflector $R$ on $\E$. If $\mathcal{S}$ is an exponential ideal, then the composite
\[\xymatrix{
A^{LX}\ar[r]^(.47){\eta_{A^{LX}}} & R(A^{LX})\ar[r]^(.53){R(A^{\alpha_X})} & R(A^X)
}\]
is an isomorphism natural in $A\in\mathcal{S}$ for every $X\in\E$.
\end{lema}
\begin{proof}
Let $\alpha:1_\E\Rightarrow L$ be the unit of the reflection and $\eta:1_\mathcal{S}\Rightarrow R$ that of the coreflection. Let $A\in\mathcal{S}$ and $X\in\E$. By Proposition A4.3.1 in \citet{MR1953060},
\[
A^{\alpha_X}:A^{LX}\cong A^X.
\]
On the other hand, since the inclusion functor $\mathcal{S}\hookrightarrow\E$ is fully faithful, $\eta$ is an isomorphism; hence
\[
\eta_{A^{LX}}:A^{LX}\cong R(A^{LX}).
\]
Therefore $R(A^X)\cong R(A^{LX})\cong A^{LX}$. The naturality is trivial.
\end{proof}

\
\begin{proof}[Proof of \ref{T:fshriekpreservesatoms}]
Without loss of generality, one may think of $F^\ast$ as an inclusion functor $\mathcal{S}\hookrightarrow\E$. Let $T\in\E$ be a atomic object. Let $A,B\in\mathcal{S}$. So, by the previous lemma, there is an isomorphism $F_\ast(A^T)\cong A^{f_! T}$ natural in $A$. Hence the following bijections are natural in $A$ and $B$:
\[\begin{tabular}{c}
$A^{F_! T}\rightarrow B$ \\
\hline
$F_\ast(A^T)\rightarrow B$ \\
\hline
$A^T\rightarrow F^! B$ \\
\hline
$A\rightarrow(F^! B)_{T}$ \\
\hline
$A\rightarrow f_\ast((F^! B)_{T}).$
\end{tabular}\]
Therefore $F_! T$ is atomic in $\mathcal{S}$.
\end{proof}

\section{Atomic objects in McLarty toposes}\label{S:McLartyAtoms}

The main purpose of this section is to prove the first part of \hyperlink{T:TeorD}{Theorem D}.  Recall that a {\em McLarty topos} is a 2-valued topos, where supports split and which is precohesive over a boolean base (see \citet{MR0877866, MR0925615}).   This means that there is a string of adjuntions

\begin{equation}\label{E:McLartybis}
\vcenter{\xymatrix{
\ar@<-3.5ex>@{}[rr]|{\dashv}\ar@<-3.5ex>@{}[rr]|(.26){\dashv}\ar@<-3.5ex>@{}[rr]|(.73){\dashv} & \E\ar@/^.7pc/[d]^{\Gamma}\ar@/_2.5pc/[d]_{\Pi} & \\
& \dec(\E)\ar@/^.7pc/[u]^{\mathcal{I}}\ar@/_2.5pc/[u]_{\Lambda} &
}}
\end{equation}
with $\Pi$ fully faithful, the counit of $\mathcal I\dashv \Gamma$ monic, and such that $\Pi$ preserves finite products.

A  2-valued topos where supports split satisfies the following Nullstellensatz: ``Every object is either initial or has global elements''.  This is equivalent, by Theorem 4.32 in \citet{MR972257}, to the local set theory $\teo(\E)$ associated to $\E$ being strongly witnessed.

\citet[Corollary 4.2]{MM2024bis} proves that for a general precohesion context, as soon as the base $\mathcal S$ is Boolean, it is equivalent to $\dec(\E)$. In order for $\E$ to satisfy the Nullstellensatz above, it is enough that base topos $\mathcal S$ satisfies it (See the Motivation section of \cite{MR4928709}).  Therefore, any Grothendieck topos whose canonical geometric morphism is precohesive is McLarty. 

The Nullstellensatz has an immediate consequence in terms of tininess. Recall that \citet{MR2177301} defines an object $T$ to be tiny---or {\em small-projective}---if $\E(X,-)$ preserves colimits. For pre-sheaf toposes, being atomic is equivalent to being small-projective (see Section 1.4 in \cite{aM1999}). The following result shows this is so for terminal objects in more general toposes and will be used to prove \hyperlink{T:TeorE}{Theorem E} in the next section. 

\begin{prop}\label{P:1issmallprojectiveinNS}
If $\E$ is a topos satisfying the Nullstellensatz, then $\E(1,-)$ preserves colimits.
\end{prop}
\begin{proof}The Nullstellensatz corresponds to the local set theory $\teo(\E)$ associated to $\E$ is strongly witnessed, and by Theorem 4.31 in \citet{MR972257}, it is also witnessed. Whence, by Theorem 4.32 in \citet{MR972257}, $\E(1,-)$ preserves quotients.

Now, it is well known that in extensive categories having exactly two complemented subobjects is equivalent to requiring its representable functor preserve coproducts. That is the case for the terminal object in a topos that satisfies the Nullstellensatz (See \cite{MR4928709}).
\end{proof}

Now, in McLarty toposes, it is easily seen that for atomic $T\in\E$,  
\[
\Gamma(2_{T})\cong 2.\]
Indeed, consider the bijection on arrows
\[\begin{tabular}{c}
$X^T\rightarrow Y$ \\
\hline
$X\rightarrow Y_T$
\end{tabular}\]
Hence if $X=1$ and $Y=2$, then $2\rightarrowtail\dso(2_T)$. Whence 
\begin{equation}\label{E:Gamma2T}\Gamma(2_T)\cong2.
\end{equation}
This is generalized by \hyperlink{T:TeorD}{Theorem D} and is used in the proof of the following result.

\begin{teorema}\label{T:AtomicDecidableisTerminal}
Every atomic decidable object of a McLarty topos is terminal.
\end{teorema}

\begin{cor}\label{C:tinyimpliesconnected}
Every atomic object of a McLarty topos is connected.
\end{cor}
\begin{proof} It follows from  \ref{T:fshriekpreservesatoms} and \ref{T:AtomicDecidableisTerminal}.
\end{proof}
\begin{proof}[Proof of \ref{T:AtomicDecidableisTerminal}]
Let $\E$ be McLarty and $A\in\dec(\E)$ be atomic in $\E$. So, by \ref{T:ExpIdealsReflect}, $A$ is atomic in $\dec(\E)$. Accordingly one has the following bijection on arrows in $\dec(\E)$:
\[\begin{tabular}{c}
$2^A\rightarrow 2$ \\
\hline
$2\rightarrow 2_A$\\
\hline
$2\rightarrow\dso(2_A)$\\
\hline
$2\rightarrow 2$
\end{tabular}\]
since $\dso(2_A)\cong 2$ by \eqref{E:Gamma2T}. Hence $2^A$ has exactly four subobjects in $\dec(\E)$ since $2=1+1$; that is, every arrow $2\rightarrow 2$ is determined by the points of 2.

Now, as $\dec(\E)$ satisfies Nullstellensatz and $A$ is atomic ($A\neq 0$), there is a point $a$ such that $a\in A$. Now, since 
\[
A=\{a\}\quad\text{and}\quad\neg(A=\{a\}).
\]
are closed formulas, they correspond to points in $\subob$, which is 2-valued. Therefore, exactly one of them is true.  If $A=\{a\}$, the proof finishes there. So suppose $\neg(A=\{a\})$. Hence there is a point $b$ such that
\[
b\in A\setminus\{a\}.
\]
That is,
\[
2\cong\{a,b\}\subseteq A.
\]
But then $2^2\subseteq 2^A$. However, since $2^2$ has four different points, then $2^2$ has at least 16 subobjects in $\dec(\E)$, which is a contradiction. Therefore it is not possible to conclude that $\neg(A=\{a\})$. That is, $A=1$.
\end{proof}
\begin{teorema}\label{T:YandYsubTsamepoints}
Let $\E$ be a McLarty topos. If $T$ is an atomic object of $\E$ then $\Gamma X\cong\Gamma(X_{T})$ for every $X\in\E$.
\end{teorema}
\begin{lema}\label{L:Nomotionispossible}
Let $\E$ be a McLarty topos. If $T$ is a connected object with a point $0:1\to T$ and $A$ is decidable, then
\[
\ev^0_A:A^T\rightarrow A
\]
is an isomorphism.
\end{lema}

\begin{proof}
It is clear that there are have natural isomorphisms
\[\xymatrix{
\E\rrtwocell^1_{-\times 1}{\varphi} & & \E
}\]
and
\[\xymatrix{
& \E\ar[dr]^{\Pi_0} & \\
\E\times\E\ar[ur]^\times\ar[dr]_{\Pi_0\times\Pi_0}\rrtwocell<\omit>{\psi} & & \E. \\
& \E\times\E\ar[ur]_\times &
}\]
Therefore the following horizontal composites are natural isomorphisms:
\begin{equation}\label{E:natiso1}
\begin{gathered}
\xymatrix{
\E^{op}\ar[r]^{\Pi_0} & \E^{op}\rrtwocell_1^{-\times 1}{\varphi} & & \E^{op}\ar[r]^{\E(-,A)} & \set
}
\end{gathered}
\end{equation}
and
\begin{equation}\label{E:natiso2}
\begin{gathered}
\xymatrix{
& \E^{op}\times\E^{op}\ar[dr]^\times & \\
\E^{op}\times\E^{op}\ar[ur]^{\Pi_0\times\Pi_0}\ar[dr]_\times\rrtwocell<\omit>{\psi} & & \E^{op}\ar[r]^{\E(-,A)} & \set. \\
& \E^{op}\ar[ur]_{\Pi_0} &
}\end{gathered}
\end{equation}

Now, suppose $A\in\dec(\E)$ and $X$ is any object of $\E$. The following are natural bijections:
\begin{align*}
\E(X,A^T) &\cong\E(X\times T,A) & \\
\E(X\times T,A) &\cong\dec(\E)(\Pi (X\times T),A). &\\
\dec(\E)(\Pi (X\times T),A)&\cong\E(\Pi_0(X\times T),A) & \\
\E(\Pi_0(X\times T),A) &\cong\E(\Pi_0 X\times\Pi_0 T,A) &\text{by }\eqref{E:natiso2}\\
\E(\Pi_0 X\times\Pi_0 T,A) &\cong\E(\Pi_0 X,A) &\text{by }\eqref{E:natiso1}\\
\E(\Pi_0 X,A) &\cong\dec(\E)(\Pi X,A) & \\
\dec(\E)(\Pi X,A) &\cong\E(X,A). &
\end{align*}
Therefore there is a natural isomorphism $\E(-,A^T)\cong\E(-,A)$. So, by Yoneda, $A^T\cong A$. 

More explicitly, consider arrows $f:X\rightarrow A^T$ and $g:X\rightarrow A$ in $\E$ related by this chain of isomorphisms. Successive corresponding arrows in the chain are represented by radial arrows proceeding counter-clockwise from $\hat f$ in the following diagram:
\begin{equation*}
\begin{xy}
(-14.7,20.225)*+{X\times T}="v3";
(-23.725,-7.725)*+{\Pi_0(X\times T)}="v1";%
(0,-25)*+{\Pi_0(X)\times\Pi_0(T)}="v0";
(23.725,-7.725)*+{\Pi_0(X)}="v2";%
(14.7,20.225)*+{X}="v4";
(0,0)*+{A}="v7";%
{\ar "v3"; "v1"}?*!/^3mm/{p}; 
{\ar "v1"; "v0"}?*!/^3mm/{\psi}; 
{\ar "v0"; "v2"}?*!/^3mm/{\pi_0}; 
{\ar "v4"; "v2"}?*!/_3mm/{p}; 
{\ar "v3"; "v4"}?*!/_3mm/{\pi_0}; 
{\ar "v0"; "v7"}?*!/_3mm/{}; 
{\ar "v1"; "v7"}?*!/_3mm/{}; 
{\ar "v2"; "v7"}?*!/_3mm/{}; 
{\ar "v3"; "v7"}?*!/_3mm/{\hat f}; 
{\ar "v4"; "v7"}?*!/^3mm/{g}; 
\end{xy}
\end{equation*}
Since the outer pentagon commutes, the whole diagram commutes with
\begin{equation}
\hat f=g\circ \pi_0,
\end{equation}
which by uniqueness must necessarily always hold.  Thus, to finish the proof, it remains to verify that 
\begin{equation}\label{E:Eval}
\ev_A^T=\widehat{1_{A^T}}=\ev_A^0\circ\pi_0.
\end{equation}
To this effect, recall that  $\ev^0_A=\ev^T_A\circ\langle1,0!\rangle$, so that
\begin{equation}
\ev^0_A\circ\pi_0=\ev^T_A\circ\langle\pi_0,0!\rangle.
\end{equation}
Notice that 
\begin{align*}
\Pi_0(\langle\pi_0,0!\rangle)&=\psi^{-1}\circ\langle\Pi_0(\pi_0),\Pi_0(0!)\rangle\\
&=\psi^{-1}\circ\langle\Pi_0(\pi_0),!\rangle\\
&=\psi^{-1}\circ\langle\Pi_0(\pi_0),\Pi_0(\pi_1)\rangle\\
&=\psi^{-1}\circ\langle\pi_0\circ\psi,\pi_1\circ\psi\rangle\\
&=\psi^{-1}\circ\langle\pi_0,\pi_1\rangle\circ\psi\\
&=1.
\end{align*}
This implies that $\Pi_0(\ev_A^T)=\Pi_0(\ev_A^0\circ\pi_0)$ and---since $p_A$ is an isomorphism---\eqref{E:Eval} as well, which finishes the proof. 
\end{proof}

\begin{prop}
Let $\E$ be a McLarty topos and let $T\in\E$ be an object not 0. If $\ev_0^A:A^T\cong A$ for every $A\in\dec(\E)$, then $T$ is connected.
\end{prop}
\begin{proof}
Since $\ev_0^A$ is natural in $A$, the following bijections are natural in $A$:
\[\begin{tabular}{c}
$\Pi X\times\Pi T\rightarrow A$ \\
\hline
$X\times T\rightarrow A$ \\
\hline
$X\rightarrow A^T$ \\
\hline
$X\rightarrow A$ \\
\hline
$\Pi X\rightarrow A.$
\end{tabular}\]
Hence, by Yoneda, $\Pi X\times\Pi T\cong\Pi X$ for every $X\in\E$; in particular, for $X=1$,
\[
\Pi T\cong 1\times\Pi T\cong 1. 
\]
\end{proof}

\begin{proof}[Proof of \ref{T:YandYsubTsamepoints}]
By the previous result, for every $A\in\dec(\E)$,
\[
\ev_A^0:A^T\cong A.
\]
Hence the following bijections are natural in $A$:
\[\begin{tabular}{c}
$A\rightarrow\Gamma Y$ \\
\hline
$\mathcal{I}A\rightarrow Y$ \\
\hline
$\mathcal{I}A^T\rightarrow Y$ \\
\hline
$\mathcal{I}A\rightarrow Y_{T}$ \\
\hline
$A\rightarrow\Gamma(Y_{T})$
\end{tabular}\]
Therefore, by Yoneda, $\Gamma Y\cong\Gamma(Y_{T})$.
\end{proof}

\begin{teorema}
Let $\E$ be a McLarty topos. If $T$ is an atomic object of $\E$ then 
\[\Pi(X_{!}):\Pi(X_{T})\rightarrow \Pi(X)\]
is an isomorphism for every $X\in\E$.
\end{teorema}
\begin{proof}
Let $T$ be an atomic object of $\E$. By \ref{C:tinyimpliesconnected}, $T$ is connected; hence $T$ has a point $0:1\rightarrow T$. Therefore, there are arrows
\[\xymatrix{
X\ar[r]^(.42){X_{0}} & X_{T}\ar[r]^(.58){X_{!}} & X\ar@{}|{=}[r] & X\ar[r]^1 & X.
}\]
Whence $\Pi(X_{0}):\Pi X\rightarrow\Pi(X_{T})$ is a split monic. That is, $\Pi X$ is a subobject of $\Pi(X_{T})$ in $\dec(\E)$. So, as $\dec(\E)$ is a topos (see Theorem C in \cite{MR4928709}), $\Pi X$ has a complement $\overbar{B}$. Therefore, since $\E$ is extensive, one has the following diagrams of pullbacks:
\[\xymatrix{
A\ar@{(>->}[rr]^(.47){i_A}\ar[d]_{(p_{X_{T}})^{-1}(\Pi X)} & & A+B\ar[d]^{p_{X_{T}}} & & B\ar@{(>->}[ll]_(.45){i_B}\ar[d]^{(p_{X_{T}})^{-1}(\overbar{B})} \\
\Pi X\ar@{(>->}[rr]_(.42){\Pi(X_{0})} & & \Pi(X_{T}) & & \overbar{B}\ar@{(>->}[ll]^(.4){i_{\overbar{B}}}
}\]
with $X_{T}=A+B$.

Suppose $B\neq 0$; so let $b':1\rightarrowtail B$ be a point, and let $b:=i_B\circ b'$. Let $r':1^T\rightarrow X$ be the transpose of $b$; that is, the following diagram commutes:
\[\xymatrix{
1^T\ar@/^.7pc/[dr]^{r'}\ar[d]_{b^T} & \\
(X_{T})^T\ar[r]_(.62){\varepsilon_{T,X}} & X.
}\]
Let $r$ be the following composite arrow:
\[\xymatrix{
1\ar[r]^(.4){\varepsilon^{-1}_{1,1}}_\cong\ar@/_1.5pc/[drrr]_r & (1_{1})^1\ar[r]^{(1_{1})^{!}}_\cong & (1_{1})^T\ar[r]^(.65)\cong & 1^T\ar[d]^{r'} \\
& & & X.
}\]
The following diagram commutes:
\[\xymatrix{
 & 1^T\ar@/_.7pc/[dl]_(.4)\cong\ar@/^6pc/[ddd]^{r'} & \\
(1_{1})^T\ar[r]^{(1_{1})^0}\ar[d]_{(r_{1})^T} & (1_{1})^1\ar[d]_{(r_{1})^1}\ar[dr]^{\varepsilon_{1,1}} & \\
(X_{1})^T\ar[r]_{(X_{1})^0}\ar[d]_{(X_{0})^T} & (X_{1})^1\ar[d]_{\varepsilon_{1,X}} & 1\ar[dl]_r \\
(X_{T})^T\ar[r]_(.62){\varepsilon_{T,X}} & X &
}\]
Therefore $b=X_{0}\circ r_{1}\circ\cong$. Hence the following diagram commutes:
\[\xymatrix{
1\ar@/^.7pc/[dr]^b\ar[d]_{r_{1}\circ\cong} & \\
X\ar[r]_{X_{0}}\ar[d]_{p_X} & X_{T}\ar[d]^{p_{X_{T}}} \\
\Pi X\ar[r]_(.4){\Pi(X_{0})} & \Pi(X_{T}).
}\]
So there is a unique arrow $1\rightarrow A$ making the following diagram commute:
\[\xymatrix{
1\ar@/^.7pc/[drr]^b\ar@/_.7pc/[ddr]_{p_X\circ r_{1}\circ\cong}\ar@{-->}[dr] & & \\
& A\ar[r]_{i_A}\ar[d] & X_{T}\ar[d]^{p_{X_{T}}} \\
& \Pi X\ar[r]_(.4){\Pi(X_{0})} & \Pi(X_{T}),
}\]
which is a contradiction, and since $\teo(\E)$ is complete and consistent, $B=0$. Therefore $\Pi(X_{0})$ is epic, from which the conclusion follows.
\end{proof}

It is straightforward to verify that for any $X$, the arrow $X_0$ is the transpose of $\ev_X^0$ with respect to the adjunction $(-)^T\dashv(-)_T$, that is.

\[
X_0=(\ev_X^0)_T\circ\eta_X.
\]
and 
\[
\ev_X^0=\varepsilon_X\circ(X_0)^T.
\]

\begin{teorema}\label{T:Ttiny.TcontractibleiffnAiso}
Let $\E$ be a McLarty topos. Let $T$ be an atomic object of $\E$. Then $T$ is contractible iff $A_0:A\rightarrow A_T$ is an isomorphism for every $A\in\dec(\E)$.
\end{teorema}
\begin{proof}
By \ref{C:tinyimpliesconnected}, $T$ is connected. Hence
\[
\ev_A^0:A^T\cong A
\]
by \ref{L:Nomotionispossible}.

Suppose $T$ is contractible; that is, by Theorem 1.4 in \cite{MR4922622},
\[
\Pi(\ev_X^0):\Pi(X^T)\rightarrow\Pi X
\]
is an isomorphism for every $X\in\E$ (the inverse of $\sigma_X$ is $\ev_X^0$). Therefore one has the following natural isomorphism:
\[\begin{tabular}{c}
$X\rightarrow\mathcal{I}A$ \\
\hline
$\Pi X\rightarrow A$ \\
\hline
$\Pi(X^T)\rightarrow A$ \\
\hline
$X^T\rightarrow\mathcal{I}A$ \\
\hline
$X\rightarrow(\mathcal{I}A)_T$
\end{tabular}\]
This yields the following commutative diagram of corresponding arrows:
\[\xymatrix{
& X\ar[dl]_{p_X}\ar[dr]^f & \\
\Pi X\ar[rr]^{\bar{f}} & & A & & (A_T)^T\ar[ll]_{\varepsilon_A}\\
& \Pi(X^T)\ar[ul]^{\Pi(\ev_X^0)}_\cong\ar[ur]_{\tilde{f}} & & X^T\ar[ll]^(.4){p_{X^T}}\ar[ul]_{\hat{f}}\ar[ur]_{(f')^T} &
}\]
So, if $f=1_A$ for $A\in\dec(\E)$, then the diagram
\[\xymatrix{
& A\ar[dl]_{p_A}\ar[dr]^1 & \\
\Pi A\ar[rr]^{p_A^{-1}} & & A & & (A_T)^T\ar[ll]_{\varepsilon_A}\\
& \Pi(A^T)\ar[ul]^{\Pi(\ev_A^0)}_\cong\ar[ur]_{\tilde{1}} & & A^T\ar[ll]^(.4){p_{A^T}}\ar[ul]_{\ev_A^0}\ar[ur]_{(A_0)^T} &
}\]
commutes, and, by Yoneda, $A_0$ is an isomorphism.

Conversely, suppose $A_0:A\rightarrow A_T$ is an isomorphism for every $A\in\dec(\E)$. Therefore one has the following natural isomorphism:
\[\begin{tabular}{c}
$\Pi X\rightarrow A$ \\
\hline
$X\rightarrow\mathcal{I}A$ \\
\hline
$X\rightarrow(\mathcal{I}A)_T$ \\
\hline
$X^T\rightarrow\mathcal{I}A$ \\
\hline
$\Pi(X^T)\rightarrow A$
\end{tabular}\]
This yields the following commutative diagram of corresponding arrows:
\[\xymatrix{
& (\Pi X)^T\ar[dr]^{f^T} & \\
X^T\ar[ur]^{(p_X)^T}\ar[rr]^{(\bar{f})^T}\ar[drr]_{(\tilde{f})^T}\ar@/_.7pc/[ddrr]_{\hat{f}}\ar[dd]_{p_{X^T}} & & A^T\ar[d]_{(A_0)^T}^\cong\ar@/^2.5pc/[dd]^{\ev_A^0} \\
& & (A_T)^T\ar[d]_{\varepsilon_A} \\
\Pi(X^T)\ar[rr]_{f'} & & A 
}\]
So, if $f=1_{\Pi X}$, then the diagram
\[\xymatrix{
& (\Pi X)^T\ar[dr]^{1^T} & \\
X^T\ar[ur]^{(p_X)^T}\ar[rr]^{(p_X)^T}\ar[drr]_{(\tilde{1})^T}\ar@/_.7pc/[ddrr]_{\hat{1}}\ar[dd]_{p_{X^T}} & & (\Pi X)^T\ar[d]_{((\Pi X)_0)^T}^\cong\ar@/^3pc/[dd]^{\ev_{\Pi X}^0} \\
& & ((\Pi X)_T)^T\ar[d]_{\varepsilon_{\Pi X}} \\
\Pi(X^T)\ar[rr]_{1'} & & \Pi X 
}\]
commutes. On the other hand, the following diagram commutes:
\[\xymatrix{
X^T\ar[rr]^{(p_X)^T}\ar[dr]^{\ev_X^0}\ar[dd]_{p_{X^T}} & & (\Pi X)^T\ar[dd]^{\ev_{\Pi X}^0} \\
& X\ar[dr]_{p_X} & \\
\Pi(X^T)\ar[rr]_{\Pi(\ev_X^0)} & & \Pi X
}\]
Therefore $1'=\Pi(\ev_X^0)$ since $p_{X^T}$ is epic. Now, by Yoneda, $\Pi(\ev_X^0)$ is an isomorphism; hence so is $\Pi(\sigma_X)$. In other words, $T$ is contractible.
\end{proof}

The following result is a strengthening of this and is the second part of \hyperlink{T:TeorE}{Theorem E}. 

\begin{teorema}
Let $\E$ be a McLarty topos with an atomic object $T$. Then $T$ is contractible iff $2_0:2\rightarrow 2_T$ is an isomorphism.
\end{teorema}
\begin{proof}
Suppose $T$ is contractible. By \ref{T:Ttiny.TcontractibleiffnAiso}, $A_0:A\rightarrow A_T$ is an isomorphism for every $A\in\dec(\E)$; in particular, for 2, which is decidable.

Conversely, suppose $2_0$ is an isomorphism. By the proof of \ref{L:superTnsubTarefaithful}, for every $X\in\E$
\[
\varepsilon_{T,X}\circ(X_T)^!\cong X_!
\]
Hence, for every $X\in\E$
\begin{equation}\label{D:nsigmaepsilon}
\varepsilon_{T,X}\circ(X_T)^!X_0\cong1.
\end{equation}

By \eqref{D:nsigmaepsilon} and the naturality of $(-)^!$, the following diagram commutes:
\[\xymatrix{
\Pi X\ar[r]^{\Pi(X^!)}\ar[d]^{\Pi g} & \Pi(X^T)\ar[d]^{\Pi(g^T)} \\
\Pi 2\ar[r]^(.45){\Pi(2^!)}\ar[d]^{\Pi(2_0)}\ar@/_5.7pc/[ddr]_{\cong} & \Pi(2^T)\ar[d]^{\Pi((2_0)^T)} \\
\Pi(2_T)\ar[r]_(.45){\Pi((2_T)^!)} & \Pi((2_T)^T)\ar[d]^{\Pi(\varepsilon_{T,2})} \\
& \Pi 2\ar[d]^{(p_2)^{-1}} \\
& 2
}\]
By the universality of $p:X\rightarrow\Pi X$, every $h:\Pi X\rightarrow 2$ is the composite $(p_2)^{-1}\circ\cong\circ\Pi g$ for a unique arrow $g:X\rightarrow 2$ in $\E$. Hence the previous diagram says that every $h:\Pi X\rightarrow 2$ factors through $\Pi(X^!)$ as indicated, and thus also through $\Pi(\sigma^T_X)$. 

That is, every $h:\Pi X\rightarrow 2$ can be extended to $\Pi(X^T)$. Since  $\Pi(\sigma^T_X)$ is monic, the extension is unique. And vice versa, every arrow $f:\Pi(X^T)\rightarrow 2$ can be restricted to $\Pi X$ via $\Pi(\sigma^T_X)$. So denote by $\theta(h)$ that extension of $h$, and by $\phi(f)$ that restriction of $f$. It is clear that $\phi\theta=1$.

To verify that $\theta\phi=1$, given $f:\Pi(X^T)\rightarrow 2$, one can get an arrow $g:X\rightarrow 2$ through the natural bijection (natural in $X$)
\[\begin{tabular}{c}
$X\underset{g}{\longrightarrow} 2$ \\
\hline
$X\underset{2_0\circ g}{\longrightarrow} 2_T$ \\
\hline
$X^T\underset{\varepsilon_2\circ (2_0)^T\circ g^T}{\longrightarrow} 2$ \\
\hline
$\Pi(X^T)\underset{p^{-1}\circ\Pi(\varepsilon_2\circ (2_0)^T\circ g^T)}{\longrightarrow} 2$
\end{tabular}\]
by going upwards. So, by going downdwards, one gets $f$ back, which is the composite at the end. Hence $\theta\phi(f)=f$. Therefore for $f_i:\Pi(X^T)\to 2$, if
$f_0\Pi(\sigma^T_X)=f_1\Pi(\sigma^T_X)$,
then
\begin{equation}\label{E:uniqueextension}f_0=f_1.
\end{equation}
Now, to see that $\Pi(\sigma^T_X)$ is epic, suppose it is not. Then, since $\dec(\E)$ is a Boolean topos, every subobject is complemented in $\dec(\E)$. Since $\Pi(\sigma^T_X)$ is a monic, $\Pi X$ is a complemented subobject of $\Pi(X^T)$, 
\[
\Pi(X^T)=\Pi X+(\Pi X)^c.
\]
Therefore, $\Pi(\sigma^T_X)$ is epic if and only if $(\Pi X)^c$ is initial. If $\Pi X)^c$ is not initial, there are at least two different arrows $(\Pi X)^c\rightarrow 2$: one $q_0$ sending everything to $\bot$ and one $q_1$ sending everything to $\top$. Therefore given an arrow $h:\Pi X\rightarrow 2$, it has at least two different extensions to $\Pi(X^T)$, v.z. the copairs $f_i:=\{h,q_i\}$. But
\[
\{h,q_0\}\circ\Pi(\sigma^T_X)=h=\{h,q_1\}\circ\Pi(\sigma^T_X).
\] 
Therefore, by \eqref{E:uniqueextension}, $\{h,q_0\}=\{h,q_1\}$, which is a contradiction. 
\end{proof}

\section{In presheaf toposes}\label{S:Grothen}
A couple observations about atomic objects in presheaf toposes with some conditions on the exponent categories are given.
\begin{teorema}\label{T:ContractPresheaf}
If $C$ is a small category with finite products and the canonical geometric morphism $g:[C^{op},\set]\rightarrow\set$ is precohesive, i.e. if $[C^{op},\set]$ is a McLarty topos, then every atomic object of $[C^{op},\set]$ is contractible.
\end{teorema}

\begin{obse}\label{R:(j,F)connectedcomps}
Given small categories $J,J'$ and a functor $F:J'\rightarrow J$,
\[\Clim J(j,F-)=\C(j\downarrow F),\]
where $\C(j\downarrow F)$ is the set of the connected components of the comma category $(j\downarrow F)$ and $J(j,-):J\rightarrow\set$ is the representable functor of $j$. Indeed, $(u,k)\sim(u',k')$ if and only if there is a zigzag of arrows
\[\xymatrix{
k_0\ar[r]^{v_0} & k_1 & k_2\ar[l]_{v_1}\ar[r]^{v_2} & k_3 &\cdots\ar[l]_{v_3}\ar[r]^{v_{r-1}} & k_r
}\]
in $J'$, with $k_0=k$ and $k_r=k'$, and $(u_i,k_i)\in(j\downarrow F)$ for $i=1,\ldots,r-1$ such that the diagram
\[\xymatrix{
& & j\ar[dll]_u\ar[dl]^{u_1}\ar[d]^(.55){u_2}\ar[dr]^(.6){u_3}\ar[drrr]^{u'} & & &\\
Fk \ar[r]_{Fv_0} & Fk_1 & Fk_2\ar[l]^{Fv_1}\ar[r]_{Fv_2} & Fk_3 &\cdots\ar[l]^{Fv_3}\ar[r]_{Fv_{r-1}} & Fk'
}\]
commutes, where $\sim$ is the equivalence relation $R(S)$ in the previous remark. That is,
\begin{align*}
J(j,Fv_0)(u) &=u_1,\\
J(j,Fv_1)(u_2) &=u_1,\\
J(j,Fv_2)(u_2) &=u_3,\\
\vdots
\end{align*}
\end{obse}

\begin{lema}\label{L:representabletoitselfisconnected}
If $C$ is a small category with finite products, then for every object $t\in C$
\[
\Clim(C(-,t)^{C(-,t)})=1.
\]
\end{lema}
\begin{proof}
Let $a\in C$. So, by Yoneda,
\begin{align*}
C(-,t)^{C(-,t)}(a) &\cong[C^{op},\set](C(-,a)\times C(-,t),C(-,t))\\
&\cong[C^{op},\set](C(-,a\times t),C(-,t))\\
&\cong C(a\times t,t).
\end{align*}
That is, $C(-,t)^{C(-,t)}=C(-\times t,t)$. Hence, by \ref{R:(j,F)connectedcomps},
\[
\Clim C(-\times t,t)=\C(-\times t\downarrow t).
\]
Let $f:a\times t\rightarrow t$ and $g:b\times t\rightarrow t$ be arrows in $C$. The following triangle diagrams commute:
\[\xymatrix{
a\times t\ar[dr]_f\ar[r]^{\pi_2^{-1}\circ f} & 1\times t\ar[d]^(.4){\pi_2} & b\times t\ar[dl]^g\ar[l]_{\pi_2^{-1}\circ g} \\
& t &
}\]
Therefore $\Clim C(-\times t,t)=1$.
\end{proof}

\begin{lema}\label{L:retractofcontcont}
Let $\E$ be a cartesian closed category with a connectedness structure $p:1\Rightarrow\Pi_0$. Any retract of a $p$-contractible object of $\E$ is $p$-contractible.
\end{lema}
\begin{proof}
Let $T$ be a $p$-contractible object of $\E$. So $T$ has a point $t:1\rightarrow T$ such that $1_T\sim t!$ (see Definition 6.2 in \citet{MR4922622}). Now, let $S$ be a retract of $T$; that is, let $i:S\rightarrow T$, $r:T\rightarrow S$ be arrows in $\E$ making the following diagram commute:
\[\xymatrix{
S\ar@<.6ex>[r]^i\ar@<-1ex>[rr]_1 & T\ar@<.6ex>[r]^r & S.
}\]
Since $p$-homotopicness behaves well w.r.t. composition (see Definition 3.7 in \citet{MR4922622}), it is readily seen to be $p$-contractible: 
\[
1_S= r\circ i\sim r\circ t!\circ i=rt!. \qedhere
\]
\end{proof}

\begin{proof}[Proof of \ref{T:ContractPresheaf}]
If $X\in [C^{op},\set]$ is atomic, then $(-)^X$ preserves colimits. Hence, by Theorem 1.4.4 in \cite{aM1999}, $X$ is a retract of a representable functor $C(-,t)$ for some $t\in C$. On the other hand, by \ref{L:representabletoitselfisconnected} and Theorem 1.4(4) in \citet{MR4922622}, $C(-,t)$ is contractible. Now, by \ref{L:retractofcontcont} and Theorem 1.4 in \citet{MR4922622}, $X$ is contractible.
\end{proof}

\begin{teorema}\label{T:ContractGrothen}
If $\E$ is a Grothendieck topos whose canonical geometric morphism $g:\E\rightarrow\set$ is precohesive and $\E(1,-)$ preserves colimits, then every atomic object of $\E$ is contractible.
\end{teorema}
\begin{proof}
Without loss of generality, suppose $\E$ is a sheaf topos $\sh(C,J)$ with $J$ subcanonical: $(C,J)$ may be standard (see Theorem C2.2.8 in \cite{MR1953060}). And let $g:\sh(C,J)\rightarrow\set$ be the canonical geometric morphism.

Now, consider the canonical coverage $J_{can}$ on $\sh(C,J)$. Hence, by Proposition C2.2.7 in \cite{MR1953060}, every object of $\sh(\sh(C,J),J_{can})$ is isomorphic to a representable functor $\sh(C,J)(-,X)$ for some $X\in\sh(C,J)$, and
\[
\sh(C,J)\simeq\sh(\sh(C,J),J_{can}).
\]
More precisely, the Yoneda embedding $X\mapsto\sh(C,J)(-,X)$ is an equivalence, whose quasi-inverse is the functor that sends a sheaf $F:\sh(C,J)^{op}\rightarrow\set$ to its restriction along the Yoneda embedding $C\rightarrow\sh(C,J)$; that is, $F$ is restricted to the representables of $C$, which are in $\sh(C,J)$ since $J$ is subcanonical.

Let $h:\sh(\sh(C,J),J_{can})\rightarrow\set$ be the canonical geometric morphism. The following diagram commutes up to isomorphism:
\[\xymatrix{
\sh(C,J)\ar[dr]_{g_\ast}\ar[rr]^(.45)Y & & \sh(\sh(C,J),J_{can})\ar[dl]^{h_\ast} \\
& \set. &
}\]
Suppose $g$ is precohesive. By Theorem 3.3 in \cite{MM2024bis}, $\sh(C,J)$ is molecular over $\set$. By Theorem 14 in \cite{MR567064}, every object of $\sh(C,J)$ is a sum of molecules. At the beginning of section 5 in \cite{MR567064} it is noted that, in the $\set$ case, a molecule is an object with only two complemented subobjects: 0 and itself. Since $\sh(C,J)$ is precohesive over $\set$, which is Boolean, by Proposition 1.1 in \cite{MR4928709},
\[
g_!(\text{molecule})=1.
\]
But, by Lemma~\ref{L:representableareindecomposable}, every representable is a molecule.

Now, by Theorem 16 in \cite{MR567064}, $\sh(C,J)$ is equivalent to a category of sheaves on some site for which the constant presheaves are sheaves. More precisely, in the proof of that theorem, the site of definition is constructed explicitly as follows. A set of generators is considered. Every generator is a sum of molecules. Then, in that proof,  they take $\mathcal{M}$ to be the full subcategory determined by those molecules, and consider $\sh(\mathcal{M},J_{can}|_{\mathcal{M}})$, with $J_{can}$ the canonical topology on the topos restricted to $\mathcal{M}$. In this case, every representable functor is a molecule and, since $J$ is subcanonical, the set of generators of $\sh(C,J)$ are the representables. That is, in this case, $\mathcal{M}$ is the full subcategory of the representable functors. By the Comparison Lemma (C2.2.3 in \cite{MR1953060}),
\[
\sh(\sh(C,J),J_{can})\simeq\sh(\mathcal{M},J_{can}|_{\mathcal{M}}),
\]
and if $k:\sh(\mathcal{M},J_{can}|_{\mathcal{M}})\rightarrow\set$ is the canonical geometric morphism, the following diagram commutes:
\[\xymatrix{
\sh(C,J)\ar[dr]_{g_\ast}\ar[rr]^(.45)Y & & \sh(\mathcal{M},J_{can}|_\mathcal{M})\ar[dl]^{k_\ast} \\
& \set. &
}\]

Therefore, $k^\ast$ is the functor which sends a set $I$ to the constant functor $\hat{I}$. Hence $k_!$ is the colimit functor (cf. the argument for (iii)$\Rightarrow$(i) in the proof of Proposition 1.3 in \cite{MR2805745}).

Notice that, by C2.2.7 in \cite{MR1953060}, every sheaf in $\sh(\mathcal{M},J_{can}|_\mathcal{M})$ is of the form $\sh(C,J)(-,X)$ for some $X\in\sh(C,J)$, restricted to $\mathcal{M}$. Hence by Yoneda, 
\[\sh(C,J)(C(-,c),X)\cong Xc,\] 
i.e. one may think of the sheaves in $\mathcal{M}$ as the sheaves in $\sh(C,J)$.

Now, let $F$ be an atomic sheaf in $\sh(C,J)$. Hence $(-)^F$ preserves colimits. On the other hand,
\[
F=\Clim_{(a,x)\in\int F}C(-,a),
\]
where $\int F$ is the category of elements of $F$ (that is, $x\in Fa$). 

Since $\sh(C,J)$ is a McLarty topos, and thus satisfies the Nullstellensatz required for \ref{P:1issmallprojectiveinNS},
\begin{align*}
\sh(C,J)(F,F) &\cong\sh(C,J)(1,F^F)\\
&\cong\sh(C,J)(1,(\Clim_{(a,x)\in\int F} C(-,a))^F)\\
&\cong\sh(C,J)(1,\Clim_{(a,x)\in\int F} C(-,a)^F)\\
&\cong\Clim_{(a,x)\in\int F}\sh(C,J)(1,C(-,a)^F) \\
&\cong\Clim_{(a,x)\in\int F}\sh(C,J)(F,C(-,a)).
\end{align*}
Whence
\[\sh(C,J)(F,\Clim_{(a,x)\in\int F}C(-,a))\cong\Clim_{(a,x)\in\int F}\sh(C,J)(F,C(-,a)).
\]

Now, let $\mu_{(a,x)}:C(-,a)\rightarrow\Clim_{(a,x)\in\int F}C(-,a)$ be the component in $(a,x)$ of the colimit cocone $\mu:C(-,a)\Rightarrow\Clim_{(a,x)\in\int F}C(-,a)$. Let $\mu'_{(a,x)}$ denote the following composite:
\[\xymatrix{
\sh(C,J)(F,C(-,a))\ar[d]^{\sh(C,J)(F,\mu_{(a,x)})} \\
\sh(C,J)(F,F)\ar[d]^\cong \\
\Clim_{(a,x)\in\int F}\sh(C,J)(F,C(-,a))
}\]
This is a colimit in $\set$. Thus this composite is the component in $(a,x)$ of a limit cocone in $\set$ whose colimit object is a quotient of the coproduct 
\[\coprod_{(a,x)\in\int F}\sh(C,J)(F,C(-,a)).\]
Hence for $1_F$, there is some pair $(a,x)$ such that there is a natural transformation $\xi\in\sh(C,J)(F,C(-,a))$ such that $\mu'_{(a,x)}(\xi)=1_F$; equivalently, for some $(a,x)$ there is $\zeta\in\sh(C,J)(F,C(-,a))$ such that $\mu_{(a,x)}\circ\zeta=1_F$. Therefore, $F$ is a retract of $C(-,a)$. Whence, by Lemmas~\ref{L:representabletoitselfisconnected}, \ref{L:retractofcontcont} and since $k_!$ is the colimit functor, $F$ is contractible.
\end{proof}

\begin{obse}
In a Grothendieck topos $\E$ whose canonical geometric morphism $g:\E\rightarrow\set$ is precohesive, a sheaf $F\in\E$ is connected iff $\int F$ is connected:
\begin{align*}
\Clim_{c\in C} Fc &=\Clim_{c\in C}\Clim_{(a,x)\in\int F}C(c,a)\\
&=\Clim_{(a,x)\in\int F}\Clim_{c\in C}C(c,a)\\
&=\Clim_{(a,x)\in\int F}1\\
&=\coprod_{\C(\int F)} 1,
\end{align*}
where $\C(\int F)$ is the set of connected components of $\int F$.
\end{obse}

\section{Examples}\label{S:Examples}

There are three examples considered in this section.  First, an example of a quality type with no nontrivial atomic objects. Next, an example of a necessarily sufficiently cohesive topos with an atomic object with countably infinite points. This example is also used to produce a counterexample to quotients of atomic objects being atomic. And last, an example of toposes (not necessarily of cohesion) with atomic objects $T$ where $2_T$ is isomorphic to $2$. 

\begin{teorema}\label{T:Nonontrivialidempotentatoms}
Let $E$ be the monoid with unit and a single idempotent element. The topos $\set^E$ has no atomic objects except 1.
\end{teorema}

\begin{teorema}\label{T:tinywithtwopoints}
There is a McLarty topos $\E$ with an atomic object $A$ with a countably infinite number of points.
\end{teorema}

\begin{teorema}\label{T:twosubTisotwo}
Let $C$ be a small cartesian closed category. Then for every $t\in C$ the representable functor $C(-,t)$ is a atomic object of $[C^{op},\set]$, and $2_{C(-,t)}\cong 2$, where $2=1+1$ in $[C^{op},\set]$. 
\end{teorema}
\begin{proof}[Proof of \ref{T:twosubTisotwo}]
By Proposition 1.1 in \cite{MR890028}, since $C^{op}$ has coproducts, $C(-,t)$ is atomic.

We have that
\begin{align*}
C(-,a)^{C(-,t)}(b) &=[C^{op},\set](C(-,b)\times C(-,t),C(-,a))\\
&\cong[C^{op},\set](C(-,b\times t),C(-,a))\\
&\cong C(b\times t,a)\\
&\cong C(b,a^t).
\end{align*}
Hence, by the proof of Lemma 1.4.3 in \cite{aM1999},
\begin{align*}
2_{C(-,t)}(a) &=[C^{op},\set](C(-,a)^{C(-,t)},2)\\
&\cong[C^{op},\set](C(-,a^t),2)\\
&\cong 2(a^t)\\
&=2.
\end{align*}
\end{proof}

\begin{lema}\label{L:exponentialobjectinsetE}
An exponential object $(\beta,B)^{(\alpha,A)}$ in $\set^E$ is the pair $(\beta^\alpha,B^\alpha)$, where $B^\alpha:=\set^E((0\times\alpha,\Z_2\times A),(\beta,B))$ and $\beta^\alpha:=(f\mapsto f\circ 0\times 1_A)$, and the evaluation morphism is $\ev:(\beta^\alpha,B^\alpha)\times(\alpha,A)\rightarrow(\beta,B)$ is defined as
\[\ev(f,a):=f(1,a).\]
\end{lema}
\begin{proof}
Let $P,Q$ be functors $E\rightarrow\set$. If $Q^P$ exists, then, as it is a functor $E\rightarrow\set$, by Yoneda,
\begin{align}
Q^P(\bullet) &\cong\Nat(E(\bullet,-),Q^P)\notag\\
&=\set^E(E(\bullet,-),Q^P)\notag\\
&\cong\set^E(E(\bullet,-)\times P,Q)\notag\\
&=\Nat(E(\bullet,-)\times P,Q)\notag.
\end{align}
So if $Q=(\beta,B)$ and $P=(\alpha,A)$, then $Q^P(\bullet)$ must be the set of morphisms $0\times\alpha\rightarrow\beta$ in $\set^E$, i. e. the functions $f$ in $\set$ making the following diagram commute:
\[\xymatrix{
\Z_2\times A\ar[r]^(.6)f\ar[d]_{0\times\alpha} & B\ar[d]^\beta \\
\Z_2\times A\ar[r]_(.6)f & B,
}\]
where 0 is the constant 0.

Now, for $e\in E$, one must have $Q^P(e)$ equal to
\[\xymatrix{
\Nat(E(\bullet,-)\times P,Q)\ar[rrr]^{\Nat(E(e,-)\times P,Q)} & & & \Nat(E(\bullet,-)\times P,Q).
}\]
So $Q^P(e)$ must be $f\mapsto f\circ 0\times 1_A$.

To see what the evaluation morphism is, define $\ev:B^\alpha\times A\rightarrow B$ as
\[\ev(f,a):=f(1,a).\]
Then, $\ev$ is a morphism $(\beta^\alpha,B^\alpha)\rightarrow(\beta,B)$. Indeed,
\[\xymatrix{
B^\alpha\times A\ar[r]^(.65)\ev\ar[d]_{\beta^\alpha\times\alpha} & B\ar[d]^\beta &
(f,a)\ar@{|->}[r]\ar@{|->}[d] & f(1,a)\ar@{|->}[d] \\
B^\alpha\times A\ar[r]_(.65)\ev & B &
(f\circ 0\times 1_A,\alpha a)\ar@{|->}[r] & f(0,\alpha a)
}\]
commutes since
\[\xymatrix{
\Z_2\times A\ar[r]^(.6)f\ar[d]_{0\times\alpha} & B\ar[d]^\beta \\
\Z_2\times A\ar[r]_(.6)f & B
}\]
commutes, so $\beta f(1,a)=f(0,\alpha a)$.

Now, let $g:(\varsigma,C)\times(\alpha,A)\rightarrow(\beta,B)$ be a morphism in $\set^E$, and define $\hat{g}:C\rightarrow B^\alpha$ as
\[\hat{g}c(n,a):=g(\varsigma^{1-n}c,a).\]
Then, $\hat{g}$ is a morphism $(\varsigma,C)\rightarrow(\beta^\alpha,B^\alpha)$ such that it is the unique function making the following diagram commute:
\[\xymatrix{
C\times A\ar[dr]^g\ar[d]_{\hat{g}\times 1} & \\
B^\alpha\times A\ar[r]_(.6)\ev & B.
}\]
\end{proof}

\begin{proof}[Proof of \ref{T:Nonontrivialidempotentatoms}]
Let $(X,x),(Y,y),(T,t)\in\set^E$. Suppose that $X,Y,T\neq\emptyset$ are finite and that $\dso(X,x)=1=\dso(Y,y)$. Suppose $(T,t)$ is an atomic object in $\set^E$ different from 1. Hence, by \ref{C:tinyimpliesconnected}, $\Pi(T,t)=1$, and since $\set^E$ is a quality type, $\dso(T,t)=1$; that is, $t$ has a single fixed point. So, by \ref{L:exponentialobjectinsetE},
\[
|X^t|=|X|^{2|T|-1}.
\]

Now, there is a bijective correspondence between the fixed points $h$ of $x^t$ and the arrows $k:(T,t)\rightarrow(X,x)$. So
\[
|\Fix(x^t)|=|X|^{|T|-1},
\]
where $\Fix$ assigns the fixed points of an idempotent function. Therefore
\begin{align*}
|\set^E((X,x)^{(T,t)},(Y,y))| &=|Y|^{|X^t|-|\Fix(x^t)|}\\
&=|Y|^{|X|^{|T|-1}\left(|X|^{|T|}-1\right)}.
\end{align*}

Say $(Y,y)_{(T,t)}=(Y_t,y_t)$. So, by \ref{T:YandYsubTsamepoints}, $\dso(Y,y)\cong\dso(Y_t,y_t)$; that is, $y_t$ has a single fixed point. Hence
\[
|\set^E((X,x),(Y_t,y_t))|=|Y_t|^{|X|-1}.
\]
All in all:
\begin{align*}
|X^t| &=|X|^{2|T|-1},\\
|\set^E((X,x)^{(T,t)},(Y,y))| &=|Y|^{|X|^{|T|-1}\left(|X|^{|T|}-1\right)},\\
|\set^E((X,x),(Y_t,y_t))| &=|Y_t|^{|X|-1}.
\end{align*}
Whence
\[
|Y_t|=|Y|^{|X|^{|T|-1}\left(|X|^{|T|-1}+|X|^{|T|-2}+\cdots+1\right)}.
\]
That is, the cardinality of $Y_t$ would (also) depend on that of $X$, which is a contradiction: it should depend on that of $Y$ and $T$ only. Therefore $\set^E$ has no atomic objects whose underlying set is finite.

Now suppose that $(T,t)$ is an atomic object of $\set^E$ with infinite underlying set. Consider the characteristic function $\chi$ of $T\setminus\{p\}$, where $p$ is the fixed point of $t$. Then $\chi:(T,t)\rightarrow(\Z_2,0)$. Define $\iota:\Z_2\hookrightarrow T$ as follows: $\iota(0):=p$ and $\iota(1):=q$, where $q$ is any other point than $p$ in $T$. Then $\iota\circ\chi=1$ and $\iota:(\Z_2,0)\rightarrow(T,t)$ in $\set^E$. Therefore $(\Z_2,0)$ is a retract of $(T,t)$. Hence, by \ref{T:retractoftinytiny}, $(\Z_2,0)$ is atomic, which is a contradiction, since $(\Z_2,0)$ has a finite underlying set.
\end{proof}

\begin{lema}\label{L:representableareindecomposable}
Let $C$ be a small category. For every $c\in C$, the representable functor $C(c,-):C\rightarrow\set$ is indecomposable. That is, its complemented subobjects are just the empty functor and itself.
\end{lema}
\begin{proof}
Let $Q$ be a subfunctor of $C(c,-)$ such that $Q\neq\emptyset$ and $Q\neq C(c,-)$. Then there is an $f\in Qd$ for some $d\in C$ since $Q\neq\emptyset$, and $1_c\notin Qc$ since $Q\neq C(c,-)$.

Now, remember
\[
(\neg Q)d=\{g\in C(c,d)\mid\text{ for all }f:d\rightarrow e,\, f\circ g\notin Qe\};
\]
therefore $1_c\notin (\neg Q)c$ since $f\circ 1_c=f\in Qd$. Whence $Q$ is not complemented. Hence the only complemented subfunctors of $C(c,-)$ are $\emptyset$ and $C(c,-)$.
\end{proof}

\begin{obse}\label{R:K(C)inheritspropertiesfromC}
Recall that given a category $C$, its Karoubi envelope $K(C)$ is contructed as follows: the objects of $K(C)$ are all the idempotent arrows in $C$ and an arrow $f:e\rightarrow e'$ in $K(C)$ is a commutative square in $C$
\[\xymatrix{
a\ar[r]^f\ar[dr]^f\ar[d]_e & b\ar[d]^{e'} \\
a\ar[r]_f & b.
}\]
Hence the identities in $K(C)$ are the arrows
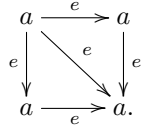
\[\xymatrix{
a\ar[r]^e\ar[dr]^e\ar[d]_e & a\ar[d]^e \\
a\ar[r]_e & a.
}\]

If $C$ has an initial object 0, then $K(C)$ has $1_0$ as initial object. Indeed, let $e:a\rightarrow a$ be any object in $K(C)$. Then the following diagram commutes:
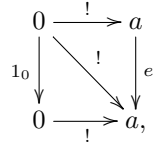
\[\xymatrix{
0\ar[r]^{!}\ar[dr]^{!}\ar[d]_{1_0} & a\ar[d]^e \\
0\ar[r]_{!} & a,
}\]
and $!$ is the only arrow making the diagram commute. Moreover, if every object of $C$ has a copoint, then every object in $K(C)$ has a copoint. Indeed, let $(a,e)\in K(C)$ and let $p:a\rightarrow 0$ be a copoint of $a$; hence the following diagram commutes:
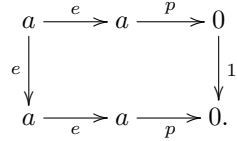
\[\xymatrix{
a\ar[r]^e\ar[d]_e & a\ar[r]^p & 0\ar[d]^1 \\
a\ar[r]_e & a\ar[r]_p & 0.
}\]
That is, $(a,e)$ has the copoint $pf:(a,e)\rightarrow(0,1_0)$.

If $C$ has the coproduct $a+b$, then $K(C)$ has the coproduct $1_{a+b}$ with injections
\[\xymatrix{
1_a\ar[r]^(.45){i_a} & 1_{a+b} & 1_b\ar[l]_(.43){i_b}.
}\]
Indeed, let $f:(a,1_a)\rightarrow(c,e)$ and $g:(b,1_b)\rightarrow(c,e)$ be arrows in $K(C)$. Hence the following diagram commutes:
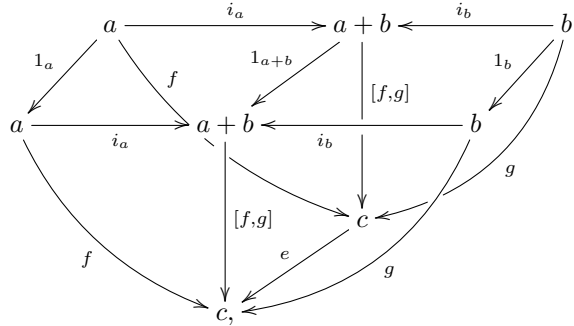
\[\xymatrix{
& a\ar[rr]^{i_a}\ar[dl]_{1_a}\ar@/_1pc/[ddrr]|(.42)\hole|(.52)\hole^(.25)f & & a+b\ar[dl]_{1_{a+b}}\ar[dd]|(.52)\hole^(.35){[f,g]} & & b\ar[ll]_{i_b}\ar[dl]_{1_b}\ar@/^1.5pc/[ddll]|(.76)\hole^g \\
a\ar[rr]_{i_a}\ar@/_1pc/[ddrr]_f & & a+b\ar[dd]^{[f,g]} & & b\ar[ll]^(.6){i_b}\ar@/^1.5pc/[ddll]^g & \\
& & & c\ar[dl]_e & & \\
& & c, & & &
}\]
and $[f,g]:a+b\rightarrow c$ is the only arrow in $C$ making the above diagram commute. Therefore $[f,g]:(a+b,1_{a+b})\rightarrow(c,e)$ is the only arrow making the following diagram commute:
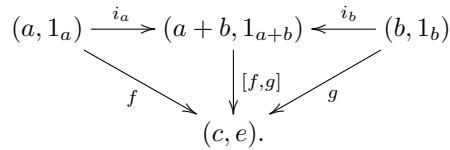
\[\xymatrix{
(a,1_a)\ar[r]^(.4){i_a}\ar[dr]_f & (a+b,1_{a+b})\ar[d]^{[f,g]} & (b,1_b)\ar[l]_(.37){i_b}\ar[dl]^g \\
& (c,e). &
}\]

If $e:a\rightarrow a$ is a split idempotent arrow in $C$ with splitting
\[\xymatrix{
a\ar[r]^r & b\ar[r]^s & a\ar@{}|{=}[r] & a\ar[r]^e & a \\
b\ar[r]^s & a\ar[r]^r & b\ar@{}|{=}[r] & b\ar[r]^{1_b} & b,
}\]
then the idempotent arrow $e:(a,1_a)\rightarrow(a,1_a)$ in $K(C)$ has the splitting $r:(a,1_a)\rightarrow(b,1_b)$, $s:(b,1_b)\rightarrow(a,1_a)$.
\end{obse}

\begin{proof}[Proof of \ref{T:tinywithtwopoints}]
Consider the underlying graph of $\finord$, the small category of all finite ordinals. Remove its terminal object and all the arrows whose domain or codomain is 1. That is, consider the vertex-deleted subgraph $\finord - 1$ of $\finord$ (we are denoting the underlying graph of $\finord$ with the same name).

Adjoin to $\finord - 1$ the following arrows (paths) with the following relations. One arrow $p:2\rightarrow 0$ with $p!=1_0$. Now, let $m\geq 2$ and let $i_m$ denote the inclusion function $2\hookrightarrow 2+m$ and $j_m$ the function embedding $m$ into $2+m$ as follows: 0 goes to 2, 1 to 3,..., $m$ to $m+1$.

So for $g:2\rightarrow k$ and $h:m\rightarrow k$ arrows already adjoined, adjoin the following arrow in correspondence with the following diagram if $[g,h]$ is not already adjoined (for example, if $g,h$ are in $\finord - 1$, $[g,h]$ is already adjoined since it is in $\finord -1$):
\[\xymatrix{
2\ar[r]^(.4){i_m}\ar@/_.7pc/[dr]_g & 2+m\ar[d]^{[g,h]} & m\ar[l]_(.35){j_m}\ar@/^.7pc/[dl]^h \\
& k &
}\]
with the obvious relations making the triangles commute, the relation
\[
[g,h]!=!,
\]
and for every arrow $r:k\rightarrow n$ already adjoined the relation $r\circ[g,h]=[rg,rh]$ corresponding to the following diagram:
\[\xymatrix{
2\ar[r]^(.4){i_m}\ar@/_.7pc/[dr]^g\ar@/_.7pc/[ddr]_{r\circ g} & 2+m\ar[d]^(.45){[g,h]} & m\ar[l]_(.35){j_m}\ar@/^.7pc/[dl]_h\ar@/^.7pc/[ddl]^{r\circ h} \\
& k\ar[d]^r & \\
& n. & 
}\]

For example, if $f_1,f_2:2\rightarrow 2$ are arrows in $\finord$, one gets the arrows $[pf_1,pf_2]:4\rightarrow 0$ and $[pf_2,pf_1]:4\rightarrow 0$ and the relations $[pf_1,pf_2]!=1_0=[pf_2,pf_1]!$, $p\circ[f_1,f_2]=[pf_1,pf_2]$.

Let $G$ be the graph obtained by adjoining to $\finord -1$ the above arrows, $S$ the set of the above relations, $T$ the set of the relations already in the category $\finord-1$, and $R:=S\cup T$. Let $D$ be the free category generated by $G$. Hence the category $C:=D/R$ with generators $G$ and relations $R$ (see Sections II.7 and II.8 in \citet{MR0354798}) has an initial object and all of its objects have a copoint. Moreover, all of its idempotents split except the constant functions in $\finord$. However, by \ref{R:K(C)inheritspropertiesfromC}, $K(C)$ has an initial object, $(2,1_2)$ is additionable, and all of its objects has a copoint. Now, $\set^C\simeq\set^{K(C)}$. Therefore, by Proposition 4.1 in \citet{MR3263283}, the canonical $f:\set^C\rightarrow\set$ is precohesive. By Proposition 4.5 in \citet{MR3263283}, $f:\set^C\rightarrow\set$ is sufficiently cohesive.

Notice that $C$ is small since the cardinality of the set of arrows in $C$ is that of $\N$ (see \ref{R:cardinalityofC} below).

There is a functor $(-)+2:C\rightarrow C$ giving coproducts with 2 for all objects in $C$. So, by Propsition 1.1 in \citet{MR890028}, $C(2,-):C\rightarrow\set$ is an atomic object in $\set^C$. Moreover, the copoints $2\rightarrow n\rightarrow 0$ in $C$ yield a countably infinite number of points $C(0,-)\rightarrow C(2,-)$ of $C(2,-)$ since $C(0,-)$ is terminal in $\set^C$.

Finally, since $f:\set^C\rightarrow\set$ is precohesive, $f$ satisfies NS (see Motivation in \citet{MR4922622}). Therefore, by Proposition 1.1 and Theorem C in \citet{MR4928709}, and \ref{L:representableareindecomposable}, $C(2,-)$ is connected.
\end{proof}

\begin{obse}\label{R:cardinalityofC}
Let $C$ be the category constructed in \ref{T:tinywithtwopoints}. Here are some examples of arrows constructed in the inductive procedure in its proof.

The only object $n\in C$ with no indecomposable arrow $n\rightarrow 0$ is 3 since 1 is not in $C$. The only object $n\in C$ with just one indemposable arrow $n\rightarrow 0$ is 2, namely, $p$. The first objet $n\in C$ with lots of indecomposable arrows $n\rightarrow 0$ is 4. The number of arrows $2\rightarrow 2\rightarrow 0$ is $2^2$. So if one considers the following diagram with $g$ and $h$ one of those arrows:
\[\xymatrix{
2\ar[r]^(.4){i_2}\ar@/_.7pc/[dr]_g & 2+2\ar[d]^{[g,h]} & 2\ar[l]_(.35){j_2}\ar@/^.7pc/[dl]^h \\
& 0, &
}\]
we get $2^2\cdot 2^2$ new $4\rightarrow 0$ arrows. With these new arrows, one gets $4^2\cdot 2^4$ new $2\rightarrow 4\rightarrow 0$ arrows if the arrows $2\rightarrow 4$ are in $\finord$. So if one considers the previous diagram with $g$ and $h$ one of these new $2\rightarrow 4\rightarrow 0$ arrows, then one gets $4^2\cdot 2^4\cdot 2^4\cdot 4^2$ new $4\rightarrow 0$ arrows, and so on.

We can do this with $5,6,7$, etc. It is clear that one gets a countably infinite number of arrows in each case.

There are arrows $n\rightarrow m$  with $m\neq 0$ not in $\finord$ not passing through 0. That occurs for $n\geq 4$. For example, let $f:2\rightarrow m$ be an arrow in $\finord$ with $m\neq 0$. Then $[!p,f]:4\rightarrow m$ does not pass through 0 and is not in $\finord$. Neither $[f,[!p,f]]:6\rightarrow m$ nor $[!p,[!p,f]]:6\rightarrow m$.
\end{obse}
\begin{teorema}\label{T:Nonatomicquotient}
Let $C$ be the category in \ref{T:tinywithtwopoints}. Let $Q:C\rightarrow\set$ be defined as follows: $Qn:=2$ for every $n\in C$ and
\[Qh:=\begin{cases}
\bar{0} &\text{if $h$ is an arrow $n\rightarrow 0\rightarrow m$ with $n\rightarrow 0$ not in $\finord$,}\\
1_2 &\text{otherwise.}
\end{cases}\]
(Here $\bar{0}$ is the constant function 0). Then, the functor $Q$ is a quotient of $C(2,-)$ that is not atomic in $\set^C$.
\end{teorema}

\begin{lema}\label{L:epicfromrep2toQ}
Let $C$ be the category in \ref{T:tinywithtwopoints}. Let $Q:C\rightarrow\set$ be as in \ref{T:Nonatomicquotient}. Then, there is an epic arrow $\alpha:C(2,-)\Rightarrow Q$ in $\set^C$.
\end{lema}
\begin{proof}
Let $\alpha:C(2,-)\Rightarrow Q$ be defined as follows: given $n\in C$
\[\alpha_n(g):=\begin{cases}
0 &\text{if $g$ is an arrow $2\rightarrow 0\rightarrow n$,}\\
1 &\text{otherwise.}
\end{cases}\]
Let $h:n\rightarrow m$ be an arrow in $C$. It is clear that the following diagram commutes:
\[\xymatrix{
C(2,n)\ar[r]^(.57){\alpha_n}\ar[d]_{C(2,h)} & Qn\ar[d]^{Qh} \\
C(2,m)\ar[r]_(.57){\alpha_m} & Qm,
}\]
and $\alpha$ is epic.
\end{proof}

\begin{lema}\label{L:reptotheQequalsrep}
Consider the functor $Q$ in \ref{L:epicfromrep2toQ}. Then $C(2,-)^Q\cong C(2,-)$, and that natural isomorphism $\Phi:(2,-)^Q\rightarrow C(2,-)$ has as components
\[\xymatrix{
\Nat(C(n,-)\times Q,C(2,-))\ar[r]^(.72){\Phi_n} & C(2,n)
}\]
\[\qquad\qquad\qquad\qquad\qquad\quad\xymatrix{
\gamma\ar@{}[r]|(.37)\mapsto & \gamma_n(1_n,1).
}\]
\end{lema}
\begin{proof}
Let $\gamma\in\Nat(C(n,-)\times Q,C(2,-))$ and $h$ be an arrow in $C$ of the form $n\rightarrow 0\rightarrow m$ with $n\rightarrow 0$ not in $\finord$. Hence, by the naturality of $\gamma$,
\begin{align*}
\gamma_m(h,0) &=h\circ\gamma_n(1_n,1)\\
&=h\circ\gamma_n(1_n,0),
\end{align*}
with $\gamma_n(1_n,1),\gamma_n(1_n,0)\in C(2,n)$. If $\gamma_n(1_n,1)\neq\gamma_n(1_n,0)$ and $h$ is arbitrary, the equation
\[
h\circ\gamma_n(1_n,1)=h\circ\gamma_n(1_n,0)
\]
should hold for any indecomposable arrow $h:n\rightarrow 0$ (and the arrow $0\rightarrow m$ would be $1_0$). So, for $n=2$ or $n\geq 4$, one has indecomposable arrows $n\rightarrow 0$. Since there are no relations equalizing $hr=hs$ for $h$ indecomposable and $r,s$ different, $\gamma_n(1_n,1)=\gamma_n(1_n,0)$ for $n=2$ or $n\geq 4$.

If $n=3$, since there are no indecomposable arrows $3\rightarrow 0$, just arrows of the form $3\rightarrow j\rightarrow 0$ with $3\rightarrow j$ in $\finord$, then it could be possible to choose an arrow $3\rightarrow j$ in $\finord$ equalizing
\[\xymatrix{
2\ar@<.5ex>[r]^r\ar@<-.5ex>[r]_s & 3
}\]
with $r,s$ in $\finord$; otherwise $3\rightarrow j$ could not equalize them. But one can choose an injective arrow $3\rightarrow j'$ instead. Therefore $\gamma_3(1_3,1)=\gamma_3(1_3,0)$.

Now, let $h':m\rightarrow k$ be another arrow in $C$ of the form $m\rightarrow 0\rightarrow k$ with $m\rightarrow 0$ not in $\finord$. Hence, by the naturality of $\gamma$,
\[
h'h\circ\gamma_n(1_n,0)=\gamma_k(h'h,0)=h'\circ\gamma_m(h,1).
\]
Therefore, by a similar argument to the previous one, it follows that
\[
h\circ\gamma_n(1_n,0)=\gamma_m(h,1).
\]

Now, suppose $h$ is not of the form $n\rightarrow 0\rightarrow m$ with $n\rightarrow 0$ not in $\finord$. Then, by the naturality of $\gamma$,
\begin{align*}
\gamma_m(h,1)=h\circ\gamma_n(1_n,1)\\
\gamma_m(h,0)=h\circ\gamma_n(1_n,0).
\end{align*}
Therefore every natural transformation $C(n,-)\times Q\Rightarrow C(2,-)$ is determined by $\gamma_n(1_n,1)=\gamma_n(1_n,0)\in C(2,n)$. That is, $C(2,-)^Q(n)\cong C(2,n)$ and that bijection is the function $(\gamma\mapsto\gamma_n(1_n,1))$.

This bijection is natural in $n$. Indeed, let $h:n\rightarrow m$ be an arrow in $C$; the following diagram commutes:
\[\xymatrix{
\Nat(C(n,-)\times Q,C(2,-))\ar[r]^(.72){\Phi_n}\ar[d]_{\Nat(C(h,-)\times Q,C(2,-))} & C(2,n)\ar[d]^{C(2,h)} \\
\Nat(C(m,-)\times Q,C(2,-))\ar[r]_(.72){\Phi_m} & C(2,m),
}\]
since
\[
\Nat(C(h,-)\times Q,C(2,-))(\gamma)=\gamma\cdot C(h,-)\times 1_Q
\]
and
\[
(\gamma\cdot C(h,-)\times 1_Q)_m(1_m,1)=\gamma_m(h,1).
\]
\end{proof}

\begin{obse}\label{R:reptotheQequalsrep}
Consider the functor $Q$ in \ref{L:epicfromrep2toQ}. If one substitutes $k$ for 2 with $k\geq 3$ in the previous lemma, all the arguments also hold for the new case $k$, even for $k=3$, which is the pathological one. So $C(k,-)^Q\cong C(k,-)$ for $k\geq 3$ in $C$.
\end{obse}

\begin{lema}\label{L:QtotheQcardinality}
Consider the functor $Q$ in \ref{L:epicfromrep2toQ}. The cardinality of $Q^Q(n)$ is $2^\N$ for every $n\in C$.
\end{lema}
\begin{proof}
Let $\gamma\in Q^Q(n)$ for $n\in C$; that is,
\[
\gamma\in\Nat(C(n,-)\times Q,Q).
\]
Let $h$ be an arrow in $C$. Then, since $\gamma$ is natural, the following diagram commutes:
\begin{equation}\label{D:naturalityofarrowinQtotheQ}
\vcenter{\xymatrix{
C(n,n)\times Qn\ar[r]^(.65){\gamma_n}\ar[d]_{C(n,h)\times Qh} & Qn\ar[d]^{Qh} \\
C(n,m)\times Qm\ar[r]_(.65){\gamma_m} & Qm.
}}\end{equation}
Hence, for $h$ of the form $n\rightarrow 0\rightarrow m$ with $n\rightarrow 0$ not in $\finord$,
\[
\gamma_m(h,0)=0.
\]
From the commutativity of \eqref{D:naturalityofarrowinQtotheQ}, for $h$ not of that form
\[
\gamma_m(h,1)=\gamma_n(1_n,1)\qquad\text{and}\qquad\gamma_m(h,0)=\gamma_n(1_n,0).
\]
There remains the case $\gamma_m(h,1)$ for $h$ of the form $n\rightarrow 0\rightarrow m$ with $n\rightarrow 0$ not in $\finord$. Let $h'$ be an arrow $m\rightarrow k$ in $C$. So the following diagram commutes:
\begin{equation}\label{D:h1hnarrow0arrowm}
\vcenter{\xymatrix{
C(n,m)\times Qm\ar[r]^(.65){\gamma_m}\ar[d]_{C(h,h')\times Qh'} & Qm\ar[d]^{Qh'} \\
C(n,k)\times Qk\ar[r]_(.65){\gamma_k} & Qk.
}}\end{equation}
Hence if $h'$ is not of the form $m\rightarrow 0\rightarrow k$ with $m\rightarrow 0$ not in $\finord$, by the commutativity of \eqref{D:h1hnarrow0arrowm},
\begin{equation}\label{E:previousarrowdeterminesnextone}
\gamma_k(h'h,1)=\gamma_m(h,1),
\end{equation}
and if $h'$ is of that form, one has a previous case: $\gamma_k(h'h,0)=0$. Therefore, by \eqref{E:previousarrowdeterminesnextone}, the arrows $h:n\rightarrow 0$ determine the value of $\gamma_m(n\rightarrow 0\rightarrow m,1)$ for the arrows $n\rightarrow 0\rightarrow m$. For example, for $p:2\rightarrow 0$ and arrows $!_m:0\rightarrow m$ for each $m\in C$, one has
\[
\gamma_m(!_mp,1)=\gamma_0(p,1).
\]
Now, there is a countably infinite number of arrows $n\rightarrow 0$. Hence there are $2^\N$ possibilities for $\gamma_0(h,1)$ for $h:n\rightarrow 0$. Therefore
\[
|Q^Q(n)|=2^\N
\]
for every $n\in C$.

Notice that an absolute freedom for the choice of the components of $\gamma$ would imply that
\[
|Q^Q(n)|=\N\times 2^\N,
\]
but $\N\times 2^\N\cong 2^\N$.
\end{proof}
\begin{proof}[Proof of \ref{T:Nonatomicquotient}]
By Theorem III.7.1 in \citet{MR0354798}, $Q$ is the colimit of represetable functors $C(n,-)$ for objects $n\in C$. If $Q$ is atomic, then $(-)^Q$ preserves colimits. Hence, by \ref{L:reptotheQequalsrep} and \ref{R:reptotheQequalsrep}, it would be true that
\[
Q^Q\cong Q.
\]
However, by \ref{L:QtotheQcardinality}, $Q^Q(n)$ has cardinality $2^\N$, but $Qn$ has cardinality 2, which is a contradiction. Therefore $Q$ is not atomic.
\end{proof}
\bibliographystyle{plainnat}
\bibliography{../ref}
\end{document}